\documentclass[a4paper, 11pt, svgnames]{article} 
\usepackage[english]{babel}
\usepackage[T1]{fontenc}
\usepackage{./packages}
\usepackage{lineno}

\title{The two-nest ants process on triangle-series-parallel graphs} 
\author{{C\'ecile Mailler}\thanks{University of Bath, Claverton Down, Bath BA2 7AY, UK} \and {Zo\'e Varin}\thanks{Université Paris Cité, CNRS, IRIF, F-75013, Paris, France}} 
\date{}

\graphicspath{{./images/}}

\begin{document}
	\maketitle
	
	\begin{abstract}
		The ants process is a stochastic process introduced by Kious, Mailler and Schapira as a model for the phenomenon of ants finding shortest paths between their nest and a source of food (seen as two marked nodes in a finite graph), with no other means of communications besides the pheromones they lay behind them as they explore their environment.
		The ants process relies on a reinforcement learning mechanism.
		In this paper, we modify the ants process by having more than one ants nest 
		(and still one source of food).
		For technical reasons, we restrict ourselves to the case when there are two nests, and when the graph is a triangle between the two nests and the source of food, whose edges have been replaced by series-parallel graphs.
		In this setting, using stochastic approximation techniques, comparison with P\'olya urns, and combinatorial arguments, we are able to prove that the ants process converges and to describe its limit.
	\end{abstract}

	\section{Introduction}
	
	The ants process, originally introduced by Kious, Mailler and Schapira~\cite{kious2020finding,kious2021tracereinforced},
	is a stochastic process of successive random walks on a graph, designed to model the well-known phenomenon of a colony of ants being able to find shortest paths between the nest and a source of food.
	It is widely believed that this biological phenomenon can be modelled using reinforcement learning, and this has inspired a whole branch of computer science called ``ant colony optimisation'' (see for example the book of Dorigo and St\"{u}tzle~\cite{Dorigo2004} for an introduction to this topic).
	
	Another striking biological phenomenon that is believed to be well-modelled by reinforcement learning is the phenomenon of a slime mould called ``physarum'' being capable of finding ``optimal transport networks''. See, for example, Tero et al.~\cite{tero2010rules} in which the authors make an experiment that shows that ``{\it the slime mold Physarum polycephalum forms networks with comparable efficiency, fault tolerance, and cost to those of real-world infrastructure networks—in this case, the Tokyo rail system}'', and introduce a mathematical model for this phenomenon.
	See also Bonifaci, Mehlhorn and Varma~\cite{bonifaci2012physarum}, in which the authors prove that, in the model of Tero et al.~\cite{tero2010rules}, physarum indeed can find shortest paths between two points in a graph.
	
	Our aim in this paper is to generalise the ants process of Kious, Mailler, and Schapira~\cite{kious2020finding,kious2021tracereinforced} 
	and see if this could be a model for the phenomenon of 
	physarum finding ``optimal'' transport networks 
	(we write ``optimal'' in quotes to acknowledge the fact 
	that there is no clear mathematical definition of that concept; this will thus be discussed later on).
	Although we are able to prove that our new ants process converges, we fail to interpret the limit as an ``optimal transport network'', suggesting that this generalisation is \emph{not} a good model for physarum.

	\subsection{The multi-nest ants process}\label{sub:def}
	Let $G = (V, E)$ be a finite graph and let $\nest_1$, $\nest_2$ and $\food$ be three vertices in~$V$.
	Let $\alpha\in [0,1]$ (this models the proportion of ants living in the first nest).
	The ants process is a process $({\bf W}(n))_{n\geq 0}$ such that, for all $n\geq 0$, ${\bf W}(n)= (W_e(n))_{e\in E}$ is a sequence of non negative integers. For all $n\geq 0$ and $e\in E$, we call $W_e(n)$ the weight of edge~$e$ at time~$n$ and interpret it as the amount of pheromones laid on that edge up to time~$n$.
	
	We let $W_e(0) = 1$ for all $e\in E$.
	Furthermore, for all $n\geq 0$, given ${\bf W}(n)$:
	\begin{enumerate}[label=(\alph*)]
		\item We set $X^{\sss (n+1)}_0 = \nest_1$ with probability $\alpha$, and $X^{\sss (n+1)}_0 = \nest_2$ with probability $1-\alpha$ (independently from everything).
		\item We let $(X^{\sss (n+1)}_i)_{i\geq 0}$ be the random walk on $V$ started at $X^{\sss (n+1)}_0$, stopped when first hitting $\food$, and whose transition probabilities are given by the edge-weights ${\bf W}(n)$. In other words, for all $i\geq 0$, for all $u\in V\setminus\{\food\}, v\in V$, the probability that $X^{\sss (n+1)}_{i+1} = v$ given that $X^{\sss (n+1)}_i = u$ is proportional to $W_{\{u,v\}}(n)$ (if $X_i^{\sss (n+1)} = \food$, then $X^{\sss (n+1)}_{i+1} = \food$ with probability~1). This is interpreted as the trajectory of the $(n+1)$-th ant.
		\item\label{item:LE} We then define $\gamma_{n+1}$ as the set of edges crossed by the $(n+1)$-th ant on its way back to the nest, and we assume that it goes back to its nest following its trajectory backwards in time, erasing loops. Equivalently, every time it sits at a vertex that was visited several times on the way forward, the ant choose to cross the edge that was crossed first on the way forward. 
		Mathematically, 
		we define $i_0 \coloneqq \min \{i : X^{\sss (n+1)}_i = \food\}$, 
		and while $j$ is such that $X^{\sss (n+1)}_{i_j} \neq X_0^{\sss (n+1)}$, 
		we define $i_{j+1}  \coloneqq \min \{i-1: X^{\sss (n+1)}_i = X^{\sss (n+1)}_{i_j} \}$. 
		Then  we set $\gamma_{n+1} = \{ \{X^{\sss (n+1)}_{i_{j+1}}, X^{\sss (n+1)}_{i_j}\}, 0 \leq j < J\}$,
		where $J$ is the largest $j$ for which $i_j$ has been defined.
		(See Figure \ref{fig:illuLEB} for an illustration.)

	\end{enumerate} 
	Finally, for all $e\in E$, we let 
	\[W_e(n+1) = W_e(n)+ {\bf 1}_{e\in \gamma_{n+1}}.\]
	
	\begin{remark}
		If $\alpha\in\{0, 1\}$, then one of the two nests is never used and we recover the model of Kious, Mailler and Schapira~\cite{kious2020finding,kious2021tracereinforced}. 
		In their papers, Kious, Mailler and Schapira study three variants of the ants process, which only differ by the definition of $\gamma_{n+1}$ as a function of $(X_i^{\sss (n+1)})_{i\geq 0}$. 
		In~\cite{kious2021tracereinforced}, they show that the ``trace-reinforced'' version of the model (in which $\gamma_{n+1}$ is the set of all edges crossed at least once by the $(n+1)$-th ant) does not find shortest paths in general, hence why we choose to not consider this version here.
		In~\cite{kious2020finding}, they consider two variants of the model, the ``loop-erased'' version and the ``geodesic'' one, which they conjecture both find shortest paths on any graph.
		We choose to consider here the ``loop-erased'' version of the model because of technical reasons and because our analysis partly relies on applying the results obtained by~\cite{kious2020finding} in the loop-erased model.
	\end{remark}
	
	\begin{figure}[t]\centering
		\begin{subfigure}[b]{0.25\textwidth}
			\centering\def\svgwidth{\columnwidth}
			\includegraphics[scale=0.65,page=1]{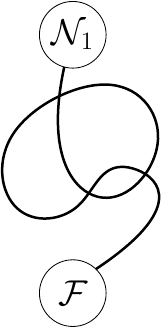}
		\end{subfigure} \hspace{1cm}
		\begin{subfigure}[b]{0.25\textwidth}
			\centering\def\svgwidth{\columnwidth}
			\includegraphics[scale=0.65,page=2]{illuLEbackward.pdf}
		\end{subfigure} 
		\caption{Illustration of the loop-erased process: $(X_i^{(n+1)})_{i\geq0}$ corresponds to the black trajectory on the leftmost picture; then the (backward) loop-erased trajectory $\gamma_{n+1}$ is represented by the orange non-dotted path on the rightmost picture.
		} 
		\label{fig:illuLEB}
	\end{figure} 
	
	\subsection{Main results}
	Our main result says that ``the multi-nest ants process finds a transport network in a triangle graph whose edges have been replaced by series-parallel graphs''. To state it, we first need a few definitions:
	\begin{definition}[See Subfigure~\ref{subfig:SPparallel} and \ref{subfig:SPseries}]
		A series-parallel graph is 
		\begin{itemize}
			\item either the graph with two vertices (called its source and its sink) linked by an edge;
			\item either two series-parallel graphs merged in series (i.e.\ the sink of the first graph is merged with the source of the second);
			\item or two series-parallel graphs merged in parallel (i.e.\ their two sources are merged into one source, and their two sinks merged into one sink).
		\end{itemize}
	\end{definition}
	
	\begin{definition}
		We define the height $h_{\min}(G)$ of a series-parallel graph $G$ as the graph distance between its source and its sink.
	\end{definition}
	
	\begin{definition}[See Subfigure~\ref{subfig:triangleSP}]\label{def:triangleSP}
		A graph $G$ with three marked nodes $\nest_1$, $\nest_2$ and $\food$ is a triangle-{SP} graph if and only if it can be obtained by replacing all three edges of the triangle graph between $\nest_1$, $\nest_2$ and~$\food$ by three series-parallel graphs. 
		
		We let $G_1 = (V_1, E_1)$ (resp.\ $G_2=(V_2, E_2)$) be the series-parallel graph replacing the edge between $\nest_1$ (resp.\ $\nest_2$) and $\food$, and $G_3 = (V_3, E_3)$ be the series-parallel graph replacing the edge between $\nest_1$ and~$\nest_2$. 
	\end{definition}

	\begin{figure}[t]\centering
		\begin{subfigure}[b]{0.25\textwidth}
			\centering\def\svgwidth{\columnwidth}
			\includegraphics[scale=0.65,page=2]{illuSP.pdf}
			\hfill
			\subcaption{ }
			\label{subfig:SPparallel}
		\end{subfigure} \hfill
		\begin{subfigure}[b]{0.25\textwidth}
			\centering\def\svgwidth{\columnwidth}
			\includegraphics[scale=0.65,page=3]{illuSP.pdf}
			\hfill
			\subcaption{ }
			\label{subfig:SPseries}
		\end{subfigure} \hfill
		\begin{subfigure}[b]{0.45\textwidth}
			\centering\def\svgwidth{\columnwidth}
			\includegraphics[scale=0.65,page=1]{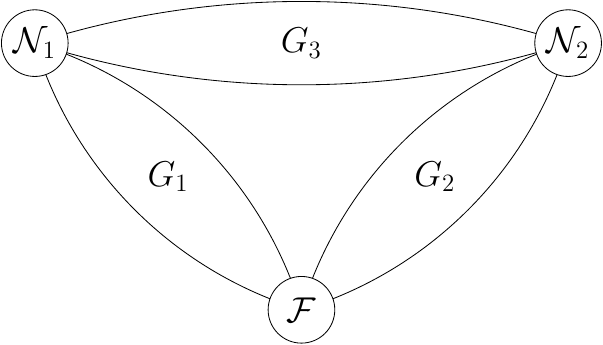}
			\hfill
			\subcaption{ }
			\label{subfig:triangleSP}
		\end{subfigure} 

		\caption{Illustration of the definition of series-parallel graphs and triangle-SP graphs. Subfigure \ref{subfig:SPparallel} represents the merging in parallel of two graphs $G_1$ and $G_2$. Similarly, Subfigure \ref{subfig:SPseries} illustrates the merging of $G_1$ and $G_2$ in series.
			Subfigure \ref{subfig:triangleSP} represents a triangle-SP graph (see Definition~\ref{def:triangleSP}). 
		} 
		\label{fig:triangleSP}
	\end{figure} 
	
	\begin{theorem}\label{thm:trianglesSP} 
		We assume that $G$ is a triangle-series-parallel graph as in Definition~\ref{def:triangleSP} and use the notation from this definition. We also assume that $\alpha\in(0,1)$. We let $\ell_i = h(G_i)$ and assume that $\ell_i\geq 1$, for all $i=1, 2, 3$.
		Finally, we let $({\bf W}(n))_{n\geq 0}$ be the (loop-erased) multi-nest ant process on~$G$.
		Then, almost surely as $n\to\infty$, for all $e\in E$,
		\[\frac{W_e(n)}{n} \to \chi_e,\]
		where $(\chi_e)_{e\in E}$ is a random vector.
		To describe $(\chi_e)_{e\in E}$, we assume, without loss of generality (by symmetry), that $\ell_1\leq \ell_2$.
		Also, we let $N_i(n)$ be the number of times up to time~$n$ when at least one edge in graph $G_i$ has been reinforced (i.e.\ its weight has increased by~1), for all $i=1,2,3$.
		
		{\bf (i)} If $\ell_2 \geq \ell_1+\ell_3$, then, almost surely as $n\to\infty$,
		\[\frac{N_i(n)}{n} \to\begin{cases}
			1 & \text{ if }i = 1\\
			0 & \text{ if }i = 2\\
			1-\alpha & \text{ if }i = 3.
		\end{cases}\]
		(Consequently, $\chi_e = 0$ for all $e\in E_2$.)
		Furthermore, for all $e\in E_1$ (resp.\ $E_3$), $\chi_e \neq 0$ if and only if the edge $e$ belongs to at least one of the shortest paths from $\nest_1$ to $\food$ inside $G_1$ (resp.\ $\nest_2$ to $\nest_1$ inside $G_3$).
		
		{\bf (ii)}  If $\ell_2 < \ell_1+\ell_3$ and $\ell_3 < \ell_1+\ell_2$, then, almost surely as $n\to\infty$,
		\[\frac{N_i(n)}{n} \to\begin{cases}
			\beta_1 & \text{ if }i = 1\\
			1-\beta_1 & \text{ if }i = 2\\
			\beta_3 & \text{ if }i = 3.
		\end{cases}\]
		where 
		\[\begin{cases}
			\beta_1 &= \displaystyle\frac{\alpha \ell_1 \left(\ell_3 + \ell_2 - \ell_1\right)}{\ell_1 \ell_3 + (\ell_2 - \ell_1 )\left((1-\alpha) (\ell_3-\ell_2) + \alpha \ell_1\right)},\\[7pt]
			\beta_3 &= \displaystyle\frac{\alpha(1-\alpha) \ell_3 (\ell_1+\ell_2-\ell_3)}{ \alpha(\ell_2 - \ell_1 )(\ell_1+\ell_2-\ell_3) + \ell_2(\ell_1-\ell_2+\ell_3)}.
		\end{cases}\]
		Moreover, for $i =1,2$, $\forall e \in E_i$ (respectively $e\in E_3$), $\chi_{e} \neq 0$ 
		if and only if $e$ belongs to at least one of the shortest paths from $\nest_i$ to $\food$ inside $G_i$ (respectively from $\nest_1$ to $\nest_2$ inside~$G_3$).
		
		{\bf (iii)} Finally, if $\ell_3 \geq \ell_1+\ell_2$, then almost surely as $n\to\infty$, 
		\[\frac{N_i(n)}{n} \to\begin{cases}
			\alpha & \text{ if }i = 1\\
			1-\alpha & \text{ if }i = 2\\
			0 & \text{ if }i = 3.
		\end{cases}\]
		(Consequently, $\forall e  \in E_3$, $\chi_{e} = 0$.) 
		Moreover, for all $i=1,2$, for all $e\in E_i$,	 
		$\chi_e \neq 0$ if and only if~$e$ belongs to at least
		one of the geodesics from $\nest_i$ to $\food$ inside $G_i$.
	\end{theorem}

	Before discussing this result in more details, we make the following elementary remarks:
	\begin{itemize}	
		\item In the case  $\ell_2 < \ell_1+\ell_3$ and $\ell_3 < \ell_1+\ell_2$, one can prove that $\beta_1,\beta_3 \in (0,1)$ (see Remark~\ref{rem:beta1andbeta3}).
		
		\item Note that we always have $\lim_{n\uparrow\infty} N_1(n)/n + \lim_{n\uparrow\infty} N_2(n)/n = 1$. This is because, by definition of the model, for all $n\geq 0$, $N_1(n) + N_2(n) = n$.
		Indeed, for all $n\geq 1$, there are only four possibilities for $\gamma_n$ (the path reinforced by the $n$-th ants): either is it a simple path from $\nest_1$ to $\food$ that only visits $G_1$, or it is a simple path from $\nest_1$ to $\food$ that only visits $G_3$ and $G_2$, or it a simple path from $\nest_2$ to $\food$ that only visits $G_2$, or it is a simple path from $\nest_2$ to $\food$ that only visits $G_3$ and $G_1$.
		
		\item Note that $(\chi_e)_{e\in E}$ remains unchanged if one multiplies $\ell_1$, $\ell_2$ and $\ell_3$ by the same constant~$c$. 
	\end{itemize}
	
	To give the ideas of the proof, we need to introduce a particular case of triangle-series-parallel graphs:
	\begin{definition}[See Figure~\ref{fig:triangle}]
		We call an SP-triangle graph the $(\ell_1, \ell_2, \ell_3)$-triangle if and only if, for all $i=1,2,3$, $G_i$ is the graph whose source is linked to the sink by $\ell_i$ edges in series.
	\end{definition}

	{\bf Ideas of the proof:} The idea of the proof of Theorem \ref{thm:trianglesSP} is that 
	the ants process restricted to each of $G_1$, $G_2$, and $G_3$ is a single-nest ants process as in the original model of Kious, Mailler and Schapira~\cite{kious2020finding}, 
	and that $(N_1(n), N_2(n), N_3(n))_{n\geq 0}$ 
	behaves like the two-nest ants process on the $(\ell_1, \ell_2, \ell_3)$ triangle.
	To prove that $\hat{\bf N}(n)=(N_1(n), N_2(n), N_3(n))/n$ converges almost surely, we show that $(\hat{\bf N}(n))_{n\geq 0}$ is a stochastic approximation and use the ODE method (see Bena\"im~\cite{benaim2006dynamics}, and~\cite{pemantle2007survey} for a survey) 
	to show its convergence.
	Intuitively, it is clear that the two-nest ants process restricted to $G_1$ and $G_2$ behaves like a single nest ant process on these series-parallel graphs, and thus, by~\cite{kious2020finding}, the ants find shortest paths within $G_1$ and $G_2$. 
	Because, when they explore $G_3$, the ants can go from $\nest_1$ to $\nest_2$ as well as from $\nest_2$ to $\nest_1$, the fact that the two-nest ant process restricted to $G_3$ is equal in distribution to a single-nest ants process is less straightforward, 
	but true, implying that the ants also find shortest paths within~$G_3$.
	
	\medskip
	Since, inside each of $G_1$, $G_2$, and $G_3$, our process finds shortest paths, 
	it is enough to discuss our result in the particular case when $G$ is an $(\ell_1, \ell_2, \ell_3)$-triangle.
	We start this discussion with very simple examples of triangle graphs and show on these simple examples that an interpretation in terms of an optimal transport network does not seem possible.

	\begin{figure}
		\begin{center}
			\includegraphics[scale=0.7,page=4]{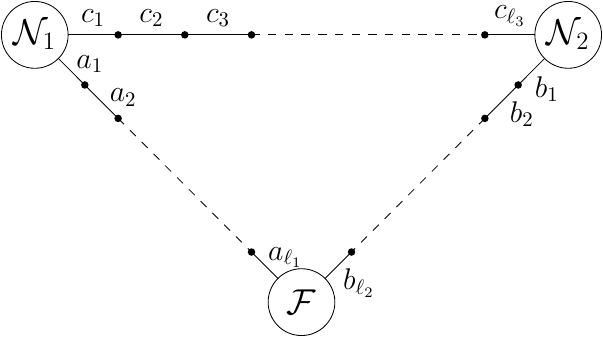}
		\end{center}
		\caption{The $(\ell_1, \ell_2, \ell_3)$-triangle.}
		\label{fig:triangle}
	\end{figure}
	
	\begin{example}
		Let $G$ be the $(1,1,2)$-triangle.
		In this case, Theorem~\ref{thm:trianglesSP} applies and shows that, if $\alpha\in (0,1)$, almost surely as $n\to\infty$,
		\[\frac{W_e(n)}n \to 
		\begin{cases}
			1&\text{ if }e = \{\nest_1, \food\}\\
			1-\alpha & \text{ if }e = \{\nest_1, \nest_2\}\\
			0 & \text{ otherwise.}
		\end{cases}\]
		This indeed corresponds to the intuition that, on $G$, ants travelling from $\nest_2$ to $\food$ (which represent a proportion $1-\alpha$) should go through $\nest_1$ and thus travel through $\{\nest_1, \nest_2\}$, while all ants (i.e.\ a proportion~1 of the colony) should travel through $\{\nest_1, \food\}$. 
		If one had to build roads for these ants to travel efficiently between $\nest_1$, $\nest_2$, and $\food$, one should build one road of ``width'' $1-\alpha$ between $\nest_1$ and $\nest_2$ (i.e.\ a road big enough to support a proportion $(1-\alpha)$ of the population), and one road of width~1 between $\nest_1$ and $\food$ (i.e.\ a road big enough to support $100\%$ of the population).
	\end{example}
	
	\begin{example}\label{ex:iso}
		Let $G$ be the $(1,1,1)$-triangle.
		In this case, Theorem~\ref{thm:trianglesSP} applies and shows that, almost surely as $n\to\infty$,
		\[\frac{W_e(n)}n \to 
		\begin{cases}
			\alpha&\text{ if }e = \{\nest_1, \food\}\\
			\alpha(1-\alpha) & \text{ if }e = \{\nest_1, \nest_2\}\\
			1-\alpha & \text{ if }e=\{\nest_2, \food\}.
		\end{cases}\]
		In this example, one would expect all the ants from $\nest_1$ to go to $\food$ through $G_1$, thus we would indeed need a road of width $\alpha$ between $\nest_1$ and $\food$.
		Similarly between $\nest_2$ and $\food$.
		The ants process seems to suggest one also builds a road of width $\alpha(1-\alpha)$ between $\nest_1$ and $\nest_2$. If one interprets $\nest_1$, $\nest_2$ and $\food$ as three cities, it would indeed make sense to build a road between $\nest_1$ and $\nest_2$ for travellers who want to go from one to the others.
		We failed at interpreting its weight of $\alpha(1-\alpha)$ as being ``optimal'' in some sense.
		
	\end{example}
	
	\begin{example}\label{ex:further}
		Let $G$ be the $(2,2,3)$-triangle.
		In this case, Theorem~\ref{thm:trianglesSP} applies and shows that, almost surely as $n\to\infty$,
		\[\frac{W_e(n)}n \to 
		\begin{cases}
			\alpha&\text{ if }e = \{\nest_1, \food\}\\
			\alpha(1-\alpha)/2 & \text{ if }e = \{\nest_1, \nest_2\}\\
			1-\alpha & \text{ if }e=\{\nest_2, \food\}.
		\end{cases}\]
		The only difference with the limit in the $(1,1,1)$-triangle case is that the weight of $\{\nest_1,\nest_2\}$ is halved; this seems to have no good interpretation if one tries to see the limit as an optimal transport network. In both examples, the shortest path between $\nest_1$ and $\nest_2$ is $\{\nest_1,\nest_2\}$, and its weight should thus be $\alpha(1-\alpha)$ in both cases; or at least the same in both cases.
	\end{example}

	We now discuss our result by applying it to a general $(\ell_1, \ell_2, \ell_3)$-triangle graph:
	
	If $\ell_2 > \ell_1+\ell_3$, our main result implies that, asymptotically, almost all ants (i.e.\ a proportion converging to~1) starting from $\nest_1$ (resp.\ $\nest_2$) go to $\food$ following edges in $G_1$ (resp.\ $G_3\cup G_1$), i.e.\ edges on the shortest path between $\nest_1$ (resp.\ $\nest_2$) and $\food$. In other words, in this case, the ants find the shortest path between their nest and the source of food.
	
	When $\ell_2 = \ell_1+\ell_3$, an interesting phenomenon occurs: 
	indeed, in this case, there are two equally long paths to go from $\nest_2$ to $\food$,
	but, asymptotically, almost all ants that start at $\nest_2$ go from $\nest_2$ to $\food$ through $G_3\cup G_1$.
	It shows that there is some interaction between the ants that start from $\nest_1$ and those starting from $\nest_2$.

	When $\ell_3 \geq \ell_1+\ell_2$, then the shortest path from $\nest_1$ 
	(resp.\ $\nest_2$) to $\food$ 
	is through $G_1$ (resp.\ $G_2$), and, asymptotically, almost no ants goes through $G_3$.
	In other words, in this case as well, the ants find the shortest paths between their nest and the source of food.

	Finally, the case $\ell_2 < \ell_1+\ell_3$ and $\ell_3 < \ell_1+\ell_2$ is more surprising: 
	since $\ell_1 \leq \ell_2$ and $\ell_2 < \ell_1+\ell_3$, 
	no shortest path from $\nest_1$ (resp.\ $\nest_2$) to $\food$ has edges in $G_3$. 
	Yet, asymptotically, a proportion $\beta_3>0$ of the ants go from their nest to the source of food through $G_3$, even though this is not an optimal route.
	The two nests are close enough to create some interaction between the ants starting at $\nest_1$ and the ants starting at $\nest_2$.
	
	\subsection{Discussion}
	{\bf Open problems:}
	Our main result leaves many questions about the multi-nest process unanswered. In particular, our main result only holds in the case when the underlying graph is a triangle-series-parallel graph, but we conjecture that the multi-nest ant process converges on any underlying graph, although it is not clear what the limit should be.
	Before looking at the fully general case, one may want to look at the $k$-cycle (or at the $k$-complete graph), with $k-1$ of the nodes being nests, and the last node being the source of food, and with each edge then replaced by a series-parallel graph.
	In this case, one would have to analyse a $(k-1)$-dimensional (resp.\ $\binom k2$-dimensional, for the $k$-complete graph) stochastic approximation. 
	
	With the aim of modelling physarum being able to find optimal transport networks, 
	one could modify the model as follows: instead of having $(k-1)$ nests and one source of food, one could imagine having $k$ distinguished nodes representing $k$ cities in the graph.
	Each ant would then pick at random one of the cities where it would start its walk (its nest), and another city where it would stop its walk (its source of food).
	Could this multi-city ants process be a better model for finding optimal transport networks?
	
	\medskip
	{\bf Other ant walks: }
	Several other probabilistic models of path formation are inspired by ants depositing pheromones behind them when exploring a graph by random walks. 
	For example, Le Goff and Raimond~\cite{LGR} consider the case of one self-avoiding vertex reinforced random walk and show that on the complete graph, and with super-linear reinforcement, the walk eventually localises on a finite number of vertices.
	Similarly, Erhard, Franco, and Reis~\cite{EFR} consider a directed-edge-reinforced random walk and show that, on any finite graph, and on $\mathbb Z^d$ ($d\geq 2$), the random walk almost surely localises on a cycle.
	More recently, Erhard and Reis~\cite{ER24} generalise the model of~\cite{EFR} to the case of $N$ walkers interacting with each other by reinforcing the same directed-edge-weights and give conditions under which the $N$ walkers all localise on the same cycle.
	
	{\medskip
		{\bf Other path-formation models:} The ants process can be seen as a path-formation model: indeed, the set of edges whose weights grow linearly in~$n$, which can be seen as the output of the process, is a subgraph of the original graph, and in some cases, a simple path.
		Just like the phenomenon of ants finding shortest paths, it is believed that the development of the brain can be modelled using reinforcement learning. 
		A probabilistic model for this phenomenon, which can also be seen as a path-formation model was proposed by van der Hofstad, Holmes, Kuznetsov and Ruszel \cite{WARM}, and later studied by
		Hirsch, Holmes and Kleptsyn~\cite{HHK, WARM3}, Holmes and Kleptsyn~\cite{HK}, Couzini\'e and Hirsch~\cite{HK}.
	}

	\paragraph{Acknowledgements}
	The authors are grateful to Daniel Kious for (early but important) discussions on this project: the project started during a visit of ZV to the University of Bath for a 5-month undergraduate research visit during which she was co-supervised by Daniel Kious and CM (who was away from Bath for the first half of ZV's visit).

	\section{Strategy and preliminary results}\label{sect:stratandpremresults}
	The proof is done in two main steps: 
	We first prove that $(N_1(n), N_2(n), N_3(n))/n$ converges almost surely when $n\to\infty$; this is done using stochastic approximation techniques.
	We then use this convergence, together with Theorem~1.3 of~\cite{kious2020finding}, to prove Theorem~\ref{thm:trianglesSP}.
	To prove convergence of $(N_1(n), N_2(n), N_3(n))/n$, we need the following results on stochastic approximations, ``$G$-urns'', and probabilities of hitting events for random walks on weighted graphs.
	
	\subsection{Random walks on weighted graphs: the method of conductances}
	
	We define the \textit{effective conductance} of a weighted series-parallel graph as in \cite{kious2020finding} (a definition for any graph is given in \cite[Chapter 2]{MR3616205} but we will not need it here). 
	The effective conductance of a series-parallel graph $G$ is 
	\begin{itemize}
		\item $C_G = w_e$, if $G$ is reduced to a single edge $e$,
		\item $G_G = C_{G_1} + C_{G_2}$, if $G$ is the composition of two \SP graphs $G_1$ and $G_2$ merged in parallel, 
		\item $G_G =\frac{1}{1/C_{G_1} + 1/C_{G_2}}$, if $G$ is the composition of two \SP graphs $G_1$ and $G_2$ merged in series.
	\end{itemize} 
	The key property of the effective conductance is that, if $G_1$ is a series-parallel graph from $S$ to $T_1$, and $G_2$ a \SP graph from $S$ to $T_2$ (as illustrated on Figure \ref{fig:illu_conductances}), then the probability that a random walk starting from $S$ reaches $T_1$ before $T_2$ is 
	\begin{equation}
		\frac{C_{G_1}}{C_{G_1} + C_{G_2}}. \label{eqn:conductances_proba}
	\end{equation}
	\begin{figure}[t]\centering
		\includegraphics[scale=0.7,page=5]{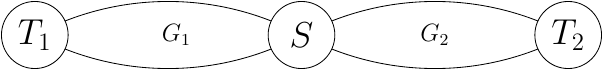}
		\caption{The probability that a random walk starting from $S$ reaches $T_1$ before $T_2$ is $\frac{C_{G_1}}{C_{G_1} + C_{G_2}}$.} \label{fig:illu_conductances}
	\end{figure}

	\subsection{Generalised Pólya urn processes}\label{sect:GPurn}
	In our proofs, we will use the following results, which are 
	stated and proved by Kious, Mailler and Schapira \cite{kious2021tracereinforced}, but can also be found in Pemantle's survey on reinforced processes~\cite{pemantle2007survey}.

	\begin{definition}\label{def:G-urn}
		Let $G$ be a function from $[0,1]$ to $[0,1]$. 
		The integer-valued process $(X_n)_{n\geq 0}$ is a \textit{$G$-urn} if, almost surely, $X_0 = 1$ and for every $n\geq 0$, $X_{n+1} \in \{X_n, X_n + 1\} $ and 
		\[\Prob{X_{n+1} = X_n + 1 \middle| X_0, \ldots, X_n} = G(\hat{X}_n)\]
		where, for all $n\geq 0$,  $\hat{X}_n = {X_n}/{n}$ (the results below are true for any normalization $\hat{X}_n = {X_n}/{(n+c)}$, with $c\geq0$).
	\end{definition}

	\begin{proposition}[\cite{pemantle2007survey} and Proposition 2.1 in \cite{kious2021tracereinforced}]\label{prop:urnstochdomin}
		Let $\left(X_n\right)_{n\geq0}$ be a $G$-urn process, with $G$ a $\mathcal{C}^1$-function. Then almost surely $\left(X_n\right)_{n\geq0}$
		converges towards a stable fixed point of $G$, that is a (possibly random) point $p\in [0, 1]$, such that $G(p) = p$
		and $G'(p) \leq 1$.
		
		In particular if there exists $c > 0$ $\eps >0$, such that $G(x) > (1+\eps)x$, for all $x \in (0, c)$
		, then almost surely $\liminf_{n \to \infty}X_n \geq c$.
	\end{proposition}
	
	\begin{corollary}[Corollary 2.2 in \cite{kious2021tracereinforced}] \label{cor:urnstochdomin}
		Let $\left(X_n\right)_{n\geq0}$ be an integer-valued process adapted to some filtration $\left(\mathcal{F}_n\right)_{n\geq0}$, such that almost surely for all $n\geq0$, $X_{n+1} \in \{X_n , X_n + 1\}$, $X_0 = 1$, and for some function $G: [0, 1] \to [0, 1]$,
		$$\Prob{X_{n+1} = X_n + 1 \middle| \mathcal{F}_n} \geq G(\hat{X}_n).$$
		If there exist $c > 0,\eps>0$, such that $G(x) > (1 + \eps)x$, for all $x \in (0, c)$, then almost surely $\liminf_{n \to \infty} \hat{X}_n \geq c$.
	
	\end{corollary}

	\subsection{The ODE method on stochastic approximations}

	In this section, we define stochastic approximations and state results from ``the ODE method'' that are useful in our proofs.
	For more information on stochastic approximation, we refer the reader to the book of Duflo~\cite{Duflo}, the book of Bena\"im~\cite{benaim2006dynamics}, and the survey of Pemantle~\cite{pemantle2007survey}.
	
	Note that there are different definitions for stochastic approximation: we choose the one that serves best our purpose:
	\begin{definition}\label{def:SA}
		A stochastic approximation is a process $(X_n)_{n\geq0}$, adapted to some filtration $(\mathcal{F}_n)_{n\geq0}$, with values in a convex compact subset $\mathcal{E} \subseteq \R^d$ ($d \geq 1$), and such that, for all $n\geq 0$,
		\begin{equation}
			X_{n+1} = X_n + \frac{F(X_n) + \xi_{n+1} + r_n}{n+1}, \text{ for all }n\geq 0, \label{eqn:def_stoch_approx}
		\end{equation}
		where the vector field $ F: \mathcal E \mapsto \R$ is some Lipschitz function, the noise $\xi_{n+1}$ is $\mathcal{F}_{n+1}$-measurable and satisfies $\E\left[\xi_{n+1}|\mathcal F_n\right]=0$, for all $n \geq 0$, and the remainder term $r_n$ is $\mathcal{F}_n$-measurable and satisfies almost
		surely { $\sum_n n^{-1} \|r(n)\| <\infty$}.
	\end{definition}

	Our next proposition states that $({\bf N}(n) = (N_1(n), N_2(n), N_3(n))_{n\geq 0}$ is a stochastic approximation. We need the following definitions first:
	for all $c = (c_1, c_2, c_3)\in [0, \infty)^3$, we let
	\begin{equation}\label{eqn:defP1}
		P_1(c) = 1-P_2(c)
		=\alpha \frac{c_1}{c_1+ \frac{c_2c_3}{c_2+ c_3}} 
		+ (1-\alpha) \frac{\frac{c_1c_3}{c_1 + c_3}}{\frac{c_1c_3}{c_1 + c_3}+c_2}
		= \frac{\alpha c_1c_2+ c_1c_3}{c_1c_2+c_1c_3+c_2c_3}.
	\end{equation}
	We also set
	\begin{equation}\label{eqn:defP3}
		P_3(c) 
		= \alpha \frac{\frac{c_2c_3}{c_2 + c_3}}{c_1 + \frac{c_2c_3}{c_2 + c_3}} + (1-\alpha) \frac{\frac{c_1c_3}{c_1 + c_3}}{\frac{c_1c_3}{c_1 + c_3}+c_2}
		= \frac{\alpha c_2c_3+(1-\alpha)c_1c_3}{c_1c_2+c_1c_3+c_2c_3}.
	\end{equation}
	Finally, for all $w\in [0, \infty)^3$, we set
	\begin{equation}\label{eqn:defp}
		p(w) = P\bigg(\frac{w_1}{\ell_1}, \frac{w_2}{\ell_2}, \frac{w_3}{\ell_3}\bigg). 
	\end{equation}

	The following lemma is useful to prove that the ants process is a stochastic approximation:
	\begin{lemma}\label{lemma:inequalities_for_the_nests} 
		Almost surely, 
		\[
		\liminf_{n\to\infty} \frac{N_1(n) + N_3(n)}{n+2} \geq \alpha 
		\quad\text{ and }\quad
		\liminf_{n\to\infty} \frac{N_2(n) + N_3(n)}{n+2} \geq 1-\alpha.\]
	\end{lemma} 
	This implies that, almost surely for all $n$ large enough,
	\begin{equation}\label{eq:liminfs}
		\frac{N_1(n) + N_3(n)}{n+2} \geq \frac\alpha2
		\quad\text{ and }\quad
		\frac{N_2(n) + N_3(n)}{n+2} \geq \frac{1-\alpha}2.
	\end{equation}
	
	\begin{proof}
		For all $n\geq 1$, we let $A(n)$ be the event that the $n$-th ant starts its walk at $\nest_1$.
		If $A(n)$ occurs then either $N_1(n+1) = N_1(n)+1$ or $N_3(n+1) = N_3(n)+1$, implying that
		\[N_1(n)+N_3(n) \geq \sum_{i=1}^n {\bf 1}_{A(i)}.\]
		Similarly,
		\[N_2(n)+N_3(n) \geq \sum_{i=1}^n \big(1-{\bf 1}_{A(i)}\big).\]
		By definition, $({\bf 1}_{A(i)})_{i\geq 1}$ is a sequence of i.i.d.\ random variables, whose distribution is the Bernoulli distribution of parameter $\alpha$.
		Thus, the conclusion follows by the strong law of large numbers, which implies that
		\[\frac1n\sum_{i=1}^n {\bf 1}_{A(i)}\to \alpha\quad\text{ and }\quad
		\frac1n\sum_{i=1}^n \big(1-{\bf 1}_{A(i)}\big) \to 1-\alpha,
		\]
		almost surely as $n\uparrow\infty$.
	\end{proof}
	
	We finally define 
		\begin{equation}\label{eq:def_E}
		\mathcal E 
		= \big\{(w_1, w_2, w_3)\in [0,1]^3\colon w_2 = 1-w_1\big\}.
	\end{equation}

	\begin{proposition}\label{prop:isastochapprox}
		For all $n\geq 1$, we let, for all $i=1,2,3$, $\hat N_i(n) = N_i(n)/n$ and $\hat{\bf N}(n) = (\hat N_1(n), \hat N_2(n), \hat N_3(n))$, where we recall that $N_i(n)$ is the number of steps up to time $n$ when the weight of at least one edge in $G_i$ increased by~1 in the ants process on a triangle-SP graph. 
		For all $n\geq 0$, we let $\mathcal F_n = \sigma({\bf W}(0), \ldots, {\bf W}(n))$.
		There exists a random integer $n_0$ such that the process $(\hat{\bf N}(n))_{n\geq n_0}$ is a stochastic approximation on 
		\begin{equation}\label{eq:def_E'}
			\mathcal E' 
			= \big\{(w_1, w_2, w_3)\in [0,1]^3\colon w_2 = 1-w_1, w_1+w_3\geq \alpha/2, \text{ and }w_2+w_3\geq (1-\alpha)/2\big\}.
		\end{equation}
		More precisely, for all $n\geq 1$,
		\begin{equation}
			{\hat{\bf N}}(n+1) 
			= {\hat{\bf N}}(n) + \frac{1}{n+1}\left(F({\hat{\bf N}}(n)) + \xi(n+1) + r(n)\right) \label{eqn:isastochapprox}
		\end{equation}
		where, for all $i=1,2,3$,
		\[F_i(w_1, w_2, w_3) = P_i\bigg(\frac{w_1}{\ell_1}, \frac{w_2}{\ell_2}, \frac{w_3}{\ell_3}\bigg)-w_i,\]
		\[\xi_i(n + 1)=  \ind{N_i (n + 1) = N_i (n) + 1} - \mathbb P(N_i (n + 1) = N_i (n) + 1|\mathcal F_n),\] 
		and 
		\[r_i(n) = \mathbb P(N_i (n + 1) = N_i (n) + 1|\mathcal F_n) - p_i(\hat{\bf N}(n)).\]
	\end{proposition} 
	
	The main difficulty in proving Proposition \ref{prop:isastochapprox} lies the following lemma, which we prove in Subsection \ref{subsect:proof_lemmeaux_isastochapprox}.
	\begin{lemma}\label{lemma:conv_rn}
		Let $(r(n))_{n\geq 1}$ be the sequence defined in Proposition~\ref{prop:isastochapprox}.
		Almost surely,
		\[\sum_{n\geq 1} n^{-1} \| r(n)\| <\infty.\]
	\end{lemma}

	\begin{proof}[Proof of Proposition \ref{prop:isastochapprox}]
		First recall that because, by definition, $\gamma_n$ is a simple path for all $n\geq 1$, we have $N_1(n) + N_2(n) = n$, and $N_3(n)\leq n$. This implies that $\hat N_i(n)\in [0,1]$ for all $i=1,2,3$, and $\hat N_1(n)+\hat N_2(n) = 1$.
		Now, for all $i=1,2,3$, for all $n\geq 1$,
		\begin{align*}\frac{N_i(n+1)}{n+1} 
			&= \frac{N_i(n)+{\bf 1}_{N_i(n+1) = N_i(n)+1}}{n+1}
			= \frac{N_i(n)}{n}\cdot \frac{n}{n+1} + \frac1{n+1} \cdot {\bf 1}_{N_i(n+1) = N_i(n)+1} \\
			& =\frac{N_i(n)}{n} + \frac1{n+1} \big({\bf 1}_{N_i(n+1) = N_i(n)+1} - \hat N_i(n)\big).
		\end{align*}
		Now note that
		\[\mathbb P(N_i(n+1)= N_i(n)+1|\mathcal F_n)= P_i({\bf C}(n)),\]
		where ${\bf C}(n) = (C_1(n), C_2(n), C_3(n))$, with $C_i(n)$ is the effective conductance of between the source and the sink of~$G_i$ when we equip the edges in $G_i$ with their weights $(W_e(n))_{e\in E_i}$ and all the other edges in $G$ with weight~0.
		We thus write
		\begin{align*}
			{\bf 1}_{N_i(n+1) = N_i(n)+1} - \hat N_i(n)
			&= p_i(\hat{\bf N}(n)) - \hat N_i(n) \\ 
			& + {\bf 1}_{N_i(n+1) = N_i(n)+1}- P_i({\bf C}(n)) \\
			& + P_i({\bf C}(n)) - p_i(\hat{\bf N}(n)),
		\end{align*}
		and set $\xi(n+1) = {\bf 1}_{N_i(n+1) = N_i(n)+1}- P_i({\bf C}(n))$ and $r(n) = P_i({\bf C}(n)) - p_i(\hat{\bf N}(n))$.
		
		Now, to conclude the proof: By Lemma \ref{lemma:inequalities_for_the_nests}, there exists $n_0$ such that, for all $n\geq n_0$, $\hat{N}(n) \in \mathcal{E}'$. 
			Furthermore, $F$ is differentiable on this compact set, and thus Lipschitz on it.

		Moreover, it is straightforward that $\E\left[\xi_{n+1}|\mathcal F_n\right]=0$, $\xi_{n+1}$ is $\mathcal{F}_{n+1}$-measurable and $r(n)$ is $\mathcal{F}_{n}$-measurable; finally, by Lemma \ref{lemma:conv_rn}, $\sum_n n^{-1} \| r(n)\| <\infty$ almost surely.
	\end{proof}
		
	The limiting set of $(X_n)_{n\geq0}$ is defined as 
	$L(X) = \cap_{n \geq 0} \overline{\cup_{k\geq n}\{X_k\}}$. 
	The following proposition is a straightforward adaptation of~\cite[Corollary 2.6]{kious2021tracereinforced} (a version of it can also be found in~\cite{pemantle2007survey}):

	\begin{proposition} \label{prop:ODEandStochApproxCV} 
		Let $(X_n)_{n\geq 0}$ be a stochastic approximation on a convex compact $\mathcal E$ (as in Definition~\ref{def:SA}).
		Assume that there exists a deterministic constant $C >0$ such that $\sup_{n\geq1}  \|\xi_n\| \leq C$ almost surely. 
		Assume also that there exist $k\geq 1$, $p_1, \ldots, p_k\in \mathcal{U} \subset \mathcal E$ such that $L(X) \subseteq \mathcal{U}$ almost surely, and for every $w \in \mathcal{U}$, the solution of the ODE $\dot{y} = {F}(y)$ started at $w$ converges to a point in $\{p_1, \ldots, p_k\}$.	
		Then, almost surely, there exists $i\in \{1,\ldots, k\}$ such that $L(X) = \{p_i\}$, i.e.\ $\lim_{n\to\infty} X_n =p_i$.
	\end{proposition}

	\subsection{The single-nest process on series-parallel graphs}\label{subsect:SPetconductances}
	
	In this section, we state results that were proved in~\cite{kious2020finding} concerning the single-nest loop-erased model.

	\begin{lemma}[See~\cite{kious2020finding}, Equation (2) and Proposition 2.5]\label{lem:bounds_KMS}
		Let $({\bf W}(n))_{n\geq 0}$ be the (single-nest) ants process on a series-parallel graph $G$ whose source is the ants' nest and sink is the ants' source of food.
		Let $C_G(n)$ be the effective conductance between $\nest$ and $\food$ in $G$ equipped with edge-weights $({\bf W}(n))_{n\geq 0}$. Then, almost surely for all $n\geq 0$
		\begin{equation}\label{eqn:bound_effective_conductance}
			\frac{n}{h_{\max}(G)} \leq C_G(n) \leq \frac{n+C}{h_{\min}(G)} 
		\end{equation}
		where $C$ is a constant 
		and $h_{\max}(G)$ is defined as the length of a longest self-avoiding path from the source to the sink in $G$.
		Also, there exists an almost-surely finite random variable~$K$ such that, almost surely, for all $n\geq 0$,
		\[C_{G}(n) \geq \frac{n - K  n^\alpha}{h_{\min}(G)},\]
		where $\alpha = h_{\min}(G)/(h_{\min}(G)+1)$.
	\end{lemma}
	
	We will apply this lemma to the two-nest process ``restricted'' to graphs $G_1$, $G_2$ and $G_3$ (the three series-parallel graphs in the triangle-SP graph of Figure~\ref{fig:triangleSP}).
	We let ${\bf W}^{\sss (i)}(n) = (W_e(n))_{e\in E_i}$, for all $i=1,2,3$, for all $n\geq 0$.
	Note that at every time step in the two-nest process on the triangle-SP graph, the set of edges whose weight increases by~1 is a simple path in either $G_1\cup G_3$, or a simple path in $G_2\cup G_3$. 
	Thus, $({\bf W}^{\sss (i)}(n))_{n\geq 0}$ is sometimes constant; we let $(\tau^{\sss (i)}_n)_{n\geq 0}$ be the increasing sequence of times at which the value of $({\bf W}^{\sss (i)}(n))_{n\geq 0}$ changes.
	
	\begin{lemma}\label{lemma:single_nest_in_multi_nest}
		For all $i=1,2,3$, the process $({\bf W}^{\sss (i)}(\tau_n^{\sss (i)}))_{n\geq 0}$ is the single-nest ants process on $G_i$ (with nest $\nest_i$ and source of food $\food$ if $i=1,2$; with nest $\nest_1$ and source of food $\nest_2$ if $i=3$).
	\end{lemma}
			
		To prove Lemma~\ref{lemma:single_nest_in_multi_nest}, we first introduce some new notation and vocabulary:
		A path in a graph $G$ is a sequence of edges.
		Given a path $\gamma$ between $S$ and $T$ in a series-parallel graph with source $S$ and sink $T$, we let $\mathrm{LE}(\gamma)$ be the path $\gamma$ obtained when erasing loops when going back from $T$ to $S$ (as in the ants process; see Section~\ref{sub:def}, item (c)).
		For any path $\gamma$, we let $\bar{\gamma}$ be the reversed path of $\gamma$, i.e.\ the same sequence of edges, but listed in reverse order.
		Finally, we say that a path is simple if it does not contain any loop.
		
		The main difficulty in the proof of Lemma~\ref{lemma:single_nest_in_multi_nest} is that, if one only looks at the trajectories of the ants restricted to $G_3$, then they sometimes start at $\mathcal N_1$ (the source of $G_3$), and sometimes start at $\mathcal N_2$ (the sink of $G_3$).
		The following lemma allows us to show that, in distribution, this does not change the path reinforced inside $G_3$:
		
		\begin{lemma}\label{claim:LEforward_egal_LEbackward} 
			Let $G$ be a series-parallel graph with source $S$, sink $T$, and edge-weights $w=(w_e)_{e\in E}\in [0,\infty)^E$.
			We let $X$ (resp.\ $Y$) 
			be a random walk on $G$, started at $S$ (resp.\ $T$) and stopped when first hitting $T$ (resp.\ $S$), with transition probabilities proportional to the edge-weights $w$.
			For all simple paths $\gamma$ from $S$ to $T$ in $G$,
			\begin{equation}
				\mathbb P(\mathrm{LE}(X) = \gamma)  = 
				\mathbb P(\mathrm{LE}(Y) = \bar\gamma), 
				\label{eqn:LEforward_egal_LEbackward}
			\end{equation}
		\end{lemma}
		
		We prove Lemma~\ref{claim:LEforward_egal_LEbackward} by induction on the number of edges in $G$. However, for the induction to work, we need that, if $G$ is two series-parallel graphs $G_1$ and $G_2$ merged in parallel or in series, then the ants process ``restricted'' to $G_1$ (resp.\ $G_2$) is also an ants process.
		In the case when $G_1$ and $G_2$ are merged in parallel, this may not be true, because a random walk in $G$ conditioned on reaching $T$ through $G_1$ typically makes more excursions away from $S$ into $G_1$ than a walk restricted to~$G_1$.
		This issue is resolved in~\cite{kious2020finding} by introducing a so-called ``generalisation'' of the model. We provide a different argument, based on the following lemma:
		
		\begin{lemma}\label{lem:Zoe}
			Let $G = (V, E)$ be a graph with two distinguished nodes $\mathcal N$ and $\mathcal F$,
			and $w = (w_e)_{e\in E}$ a sequence of non-negative weights.
			Let $X$ be a weighted random walk on a graph $G$ started at $\mathcal N$, stopped when first hitting $\mathcal F$, and with transitions probabilities proportional to the edge-weights~$w$.
			We write $X$ as a concatenation of excursions away from $\mathcal N$: 
			$X = U S$, where $U$ is the path of the random walk $X$ until it last visits $\mathcal N$, and $S$ is the path after this last visit to $\mathcal N$.
			Then, in distribution,
			\[\mathrm{LE}(X) = \mathrm{LE}(S).\]
		\end{lemma}

		\begin{figure}
			\begin{center}
				\includegraphics[width=10cm]{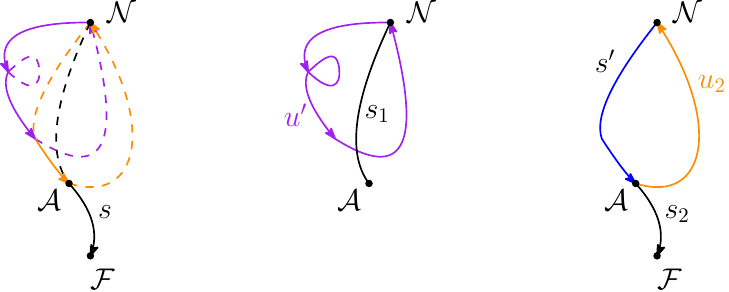}
			\end{center}
			\caption{Visual support for the proof of Lemma~\ref{lem:Zoe}. 
				On the left-hand side: the path $\gamma$ is decomposed into two parts $u$ (until the last visit at $\mathcal N$) and $s$ (after the last visit at $\mathcal N$). 
				For better readability, we have coloured the first part of $u$ in purple, the second part in orange, and $s$ in black. The non-dashed part of $\gamma$ is $\mathrm{LE}(\gamma)$.
				In the middle: the path $u' s_1$, to which we apply $\psi$ recursively in the proof of Lemma~\ref{lem:Zoe}.
				On the right-hand side: the path $u_2 s's_2$, which goes from $\mathcal A$ to $\mathcal F$.
				Recall that we define $\psi(\gamma) = \psi(u's_1) u_2s's_2$.
			}
			\label{fig:def_psi}
		\end{figure}
		
		\begin{proof}
			The idea of the proof is to define a bijection $\psi$ on the set of paths started at $\mathcal N$ and stopped at a vertex $\mathcal A$ of $G$ that is not hit before $\mathcal F$ in $\mathcal G$, such that:
			\begin{itemize} 
				\item[(i)] $\psi(\gamma)$ is a permutation of $\gamma$ in the sense that both paths are ordered lists of the same oriented edges, in different orders (implying that $\mathbb P(X = \gamma) = \mathbb P(X = \psi(\gamma))$); 
				\item[(ii)] if $\gamma = us$, where $u$ is the path $\gamma$ until it last visits $\mathcal N$, and $s$ the path $\gamma$ after it last visits $\mathcal N$,
				then $\mathrm{LE}(\psi(\gamma)) = \mathrm{LE}(s)$.
			\end{itemize}
			We prove the existence of $\psi$ by induction on $|V|$, the number of vertices in $G$.
			If $|V| = 2$, then the only two nodes in $G$ are $\mathcal N$ and $\mathcal F$; thus, any path $\gamma$ goes from $\mathcal N$ to $\mathcal F$ in one step. We set $\psi(\gamma) = \gamma$ for all $\gamma$.
			
			We now assume that $|V|\geq 3$ and that $\psi$ has been defined on any graph with at most $|V|-1$ vertices.
			We let $\gamma = u s$ be a path from $\mathcal N$ to $\mathcal F$ decomposed as in (ii) (also see Figure~\ref{fig:def_psi}).
			We define $\psi(\gamma)$ as follows: 
			we let $\mathcal A$ be the vertex at which $s$ last intersects with $u$.
			If $\mathcal A	= \mathcal N$, then we let $\psi(\gamma) = \gamma$ 
			and note that $\mathrm{LE}(\gamma) = \mathrm{LE}(s)$.
			Thus, (ii) is satisfied in this case.
			Otherwise, we let $u_1$ be the path $u$ until it first hits $\mathcal A$, and $u_2$ be the rest of~$u$. 
			We also let $s_1$ be the path $s$ until it first hits $\mathcal A$, and $s_2$ the path $s$ after it first hits $\mathcal A$.
			We thus have $\gamma = u_1u_2s_1s_2$.
			We also decompose $u_1$ as a $u_1 = u' s'$, where $u'$ is the path $u_1$ until it last visits $\mathcal N$, and $s'$ the rest of $u_1$. 
			Note that, by definition, both $u'$ and $s_1$ never hit $\mathcal F$; 
			thus, they are paths in a graph $G'$ whose vertex set is $V\setminus\{\mathcal F\}$, on which $\psi$ is well-defined, by the induction hypothesis.
			Thus, we let $\psi(\gamma) = \psi(u' s_1) u_2 s' s_2$. 
			
			To conclude the induction step, it only remains to prove that $\psi$ is bijective and satisfies (i-ii).
			
			$\bullet$ We first prove that $\psi$ is a bijection. 
			This is because, from $\eta = \psi(\gamma)$, one can recover $\psi(u's_1)$, $u_2$, $s'$ and $s_2$, and thus $\gamma$ as follows: 
			we first write $\eta = ab$, where $a$ is the path $\eta$ until the last visit at $\mathcal N$, and $b$ is the rest. 
			We then let $\mathcal A$ be the vertex where $b$ last intersects $a$. 
			If $\mathcal A = \mathcal N$, then $\gamma = \eta$.
			If $\mathcal A\neq\mathcal N$, 
			then we let $\eta_1$ be the path $\eta$ until the first visit at $\mathcal A$, and $\eta_2$ the rest of the path $\eta$. 
			By definition, $\eta_1$ is a path {that hits its final vertex only once} on a graph of vertex set $V\setminus\{\mathcal F\}$ on which, by the induction hypothesis, $\psi$ is bijective, we thus set $\gamma' = \psi^{-1}(\eta_1)$.
			We let $u'$ be the path $\gamma'$ until it last hits~$\mathcal N$, and $s_1$ be the path $\gamma'$ after it last hits $\mathcal N$.
			Finally, note that $\eta_2$ is a path from $\mathcal A$ to $\mathcal F$.
			We let $u_2$ be the path $\eta_2$ until its last visit at~$\mathcal N$, $s'$ the path $\eta_2$ from its last visit at $\mathcal N$ until its first visit at $\mathcal A$ {after that}, 
			and $s_2$ the rest of the path $\eta_2$ (thus from $\mathcal A$ to $\mathcal F$).
			Then, we get that, if $\gamma = u' s' u_2 s_1s_2$, then $\psi(\gamma) = \eta$. 
			We have thus been able to define $\psi^{-1}$, proving that $\psi$ is indeed bijective.
			
			$\bullet$ We now prove that $\psi$ satisfies (i). 
			Recall that $\psi(\gamma) = \psi(u' s_1) u_2 s' s_2$ and, 
			by the induction hypothesis, $\psi(u' s_1)$ is a permutation of $u' s_1$.
			Thus, $\psi(\gamma)$ is a permutation of $u' s_1 u_2 s' s_2$ and thus a permutation of $\gamma = u' s' u_2 s_1 s_2$.
			
			$\bullet$ Finally, we prove that $\psi$ satisfies (ii).
			First note that $\mathrm{LE}(\psi(\gamma)) = \mathrm{LE}(\psi(u's_1)) \mathrm{LE}(s_2)$. By the induction hypothesis, $\mathrm{LE}(\psi(u's_1)) = \mathrm{LE}(s_1)$, which thus implies $\mathrm{LE}(\psi(\gamma))  = \mathrm{LE}(s_1) \mathrm{LE}(s_2) = \mathrm{LE}(s_1s_2) = \mathrm{LE}(s)$.
			
			This concludes the proof.
		\end{proof}
		
		\begin{proof}[Proof of Lemma~\ref{claim:LEforward_egal_LEbackward}]
			We reason by induction on $|E|$, the number of edges of $G$.
			It $|E| = 1$, there is a unique simple path, say $\gamma_0$ between $\mathcal S$ and $\mathcal T$, and thus, 
			$\mathbb P(\mathrm{LE}(X) = \gamma_0) =
			\mathbb P(\mathrm{LE}(Y) = \bar\gamma_0)=1$, which concludes the proof of the base case.
			We now assume that $G$ has $N+1$ edges and that the induction hypothesis (i.e.\ the statement of the lemma) holds for any series-parallel graph with at most $N$ edges.
			By definition, $G$ is the merging of two non-empty series-parallel graphs in series or in parallel. We treat the two cases separately.
			
			{\bf Series case: } We first assume that $G$ can be decomposed as two series-parallel graphs $G_1 = (V_1, E_1)$ and $G_2 = (V_2, E_2)$ merged in series. 
			We call $\mathcal S_1$,  $\mathcal S_2$, $\mathcal T_1$, and $\mathcal T_2$ the respective sources and sinks of $G_1$ and $G_2$. 
			We assume that $\mathcal S = \mathcal S_1$, $\mathcal T_1 = \mathcal S_2$, and $\mathcal T = \mathcal T_2$.
			We now look at the random walk $X$: let $\tau$ be the first time $X$ hits $S_1 = T_2$ and let $X_{\leq \tau}$ be the path followed by the random walk until that time.
			Note that $X_{\leq\tau}$ is a random walk on $G_1$, started at $\mathcal S_1$, stopped when first hitting $\mathcal T_1$, and with transition probabilities proportional to the edge weights. Thus, the induction hypothesis applies to $X_{\leq \tau}$.
			Similarly, we let $X_{>\tau}$ be the path followed by $X$ after time $\tau$: 
			this path goes from $\mathcal T_1 = \mathcal S_2$ to $\mathcal T_2$ and possibly visits edges in $E_1$. 
			We let $X^{\sss (2)}_{>\tau}$ be the path obtained when removing from $X_{>\tau}$ all the edges that belong to $E_1$; by definition, $X^{\sss (2)}_{>\tau}$ is a path in $G_2$.
			By definition, $X^{\sss (2)}_{>\tau}$ is equal, in distribution, to a random walk started at $\mathcal S_2$, stopped when first hitting $\mathcal T_2$, and whose transition probabilities are proportional to the edge weights.
			Thus, the induction hypothesis also applies to $X^{\sss (2)}_{> \tau}$.
			Finally, note that
			\[\mathrm{LE}(X) = \mathrm{LE}(X_{\leq \tau})\mathrm{LE}(X^{\sss (2)}_{>\tau}).\]
			
			We now let $\gamma$ be a simple path from $\mathcal S$ to $\mathcal T$.
			There is a unique way to decompose $\gamma$ as $\gamma_1\gamma_2$ (where this denotes the concatenation of paths $\gamma_1$ and $\gamma_2$), 
			where $\gamma_1$ and $\gamma_2$ are simple paths respectively in $G_1$ and $G_2$.
			By the induction hypothesis,
			\[\mathbb P(\mathrm{LE}(X_{\leq \tau}) = \gamma_1) = \mathbb P(Y^{\sss (1)} = \bar\gamma_1)
			\quad\text{ and }\quad
			\mathbb P(\mathrm{LE}(X^{\sss (2)}) = \gamma_2) = \mathbb P(Y^{\sss (2)} = \bar\gamma_2),
			\]
			where $Y^{\sss (1)}$ (resp.\ $Y^{\sss (2)}$) is a random walk on $G_1$ (resp.\ $G_2$), started at $\mathcal T_1$ (resp.\ $\mathcal T_2$), stopped when first hitting $\mathcal S_1$ (resp.\ $\mathcal S_2$), and whose transition probabilities are proportional to the edge-weights.
			Thus,
			\begin{align*}
				\mathbb P(\mathrm{LE}(X) = \gamma) 
				&= \mathbb P(\mathrm{LE}(X_{\leq\tau}) = \gamma_1) \mathbb P(\mathrm{LE}(X^{\sss (2)}) = \gamma_2)\\
				&= \mathbb P(\mathrm{LE}(Y^{\sss (1)}) = \bar\gamma_1) \mathbb P(\mathrm{LE}(Y^{\sss (2)}) = \bar\gamma_2) = \mathbb P(\mathrm{LE}(Y) = \bar\gamma),
			\end{align*}
			where $Y$ is a random walk on $G$, starting at $\mathcal T$, stopped when first hitting $\mathcal S$, and with edge-weights $(w_e)_{e\in E}$.
			The last equality can be justified as we did the first one.
			This concludes the proof in the series case.

			{\bf Parallel case: } We now assume that $G$ can be decomposed as two series-parallel graphs $G_1 = (V_1, E_1)$ and $G_2 = (E_2, V_2)$ merged in parallel. We assume $\mathcal S = \mathcal S_1 = \mathcal S_2$ and $\mathcal T = \mathcal T_1 = \mathcal T_2$.
			Note that, if $\gamma$ is a simple path in $G$, it is either a simple path in $G_1$ or in $G_2$. Without loss of generality, we assume that $\gamma$ is a simple path in $G_1$.
			We let $X^{\sss (1)}$ be a random walk in $G_1$, started at $\mathcal S = \mathcal S_1$, stopped when first hitting $\mathcal T = \mathcal T_1$, and with edge-weights $(w_e)_{e\in E_1}$.
			By the induction hypothesis,
			\[\mathbb P(\mathrm{LE}(X^{\sss (1)}) = \gamma) = \mathbb P(\mathrm{LE}(Y^{\sss (1)}) = \bar\gamma),\]
			where $Y^{\sss (1)}$ be a random walk in $G_1$, started at $\mathcal T = \mathcal T_1$, stopped when first hitting $\mathcal S = \mathcal S_1$, and with edge-weights $(w_e)_{e\in E_1}$.
			By Lemma~\ref{lem:Zoe}, in distribution,
			$\mathrm{LE}(X) = \mathrm{LE}(S)$, 
			where $S$ is the path $X$ after its last visit to $\mathcal S$.
			Note that, by definition, $\mathcal S$ is in $G_1$ with probability $\mathcal C_{G_1}/(\mathcal C_{G_1}+\mathcal C_{G_2})$; furthermore, given that it is in $G_1$, 
			it is a random walk in $G_1$ with transition probabilities proportional to the edge-weights, 
			started at $\mathcal S =\mathcal S_1$, stopped when first hitting $\mathcal T = \mathcal T_1$, 
			conditioned on not returning to $\mathcal S_1$ before hitting $\mathcal T_1$. 
			
			Thus, for any simple path $\gamma$ from $\mathcal S$ to $\mathcal T$ in $G_1$
			\begin{equation}\label{eq:reason}
				\P(\mathrm{LE}(X) = \gamma)
				=\P(\mathrm{LE}(S) = \gamma)
				= \P(\mathrm{LE}(S^{\sss (1)}) = \gamma ) \frac{\mathcal{C}_{G_1}}{\mathcal{C}_{G_1} + \mathcal{C}_{G_2}}
				=\P(\mathrm{LE}(X^{\sss (1)}) = \gamma ) \frac{\mathcal{C}_{G_1}}{\mathcal{C}_{G_1} + \mathcal{C}_{G_2}},
			\end{equation}
			where we have let, as above,
			$X^{\sss (1)}$ be a random walk in $G_1$, started at $\mathcal S_1$, stopped when first hitting $\mathcal T_1$, and with transition probabilities proportional to the edge-weights, and $S^{\sss (1)}$ be the path of $X^{\sss (1)}$ after its last visit at $\mathcal S$.
			Now letting $Y^{\sss (1)}$ be a random walk in $G_1$, started at $\mathcal T_1$, stopped when first hitting $\mathcal S_1$, and with transition probabilities proportional to the edge-weights, we get, by the induction hypothesis,
			\[\P(\mathrm{LE}(X) = \gamma)
			= \P(\mathrm{LE}(Y^{\sss (1)}) = \bar{\gamma}) \frac{\mathcal{C}_{G_1}}{\mathcal{C}_{G_1} + \mathcal{C}_{G_2}}
			=\P(\mathrm{LE}(Y) = \bar\gamma).\]
			(For the last equality, we reason as in~\eqref{eq:reason}.)
		\end{proof}
		
		\begin{proof}[Proof of Lemma \ref{lemma:single_nest_in_multi_nest}] 
			For $i=1, 2, 3$, we let $(\bar{\bf W}^{\sss (i)}(n))_{n\geq 0}$ be the single-nest ants process on~$G_i$.
			We fix $i\in\{1, 2, 3\}$ and prove by induction on $n$ that, in distribution,
			\[\bar{\bf W}^{\sss (i)}(n) = {\bf W}^{\sss (i)}(\tau^{\sss (i)}_n).\]
			Because $i\in \{1, 2, 3\}$ is fixed, we now abuse notation by ignoring all super-indices $(i)$.
			We reason by induction on $n$: the case $n=0$ is trivial since, at time zero in both models, the edge weights are all equal to~1.
			For the induction step, it is enough to show that, if $\bar{\bf W}(n) = {\bf W}(\tau_n) = (w_e)_{e\in E_i}$, then, in distribution, 
			\[\gamma_{\tau_{n+1}}|_{G_i} = \eta_{n+1},\]
			where $\gamma_{\tau_{n+1}}|_{G_i}$ is the path $\gamma_{\tau_{n+1}}$ (the path reinforced by the multi-nest ant process at time $\tau_{n+1}$) in which we have removed all edges that do not belong to $G_i$, and where $\eta_{n+1}$ is the path reinforced by the single-nest ant process on $G_i$ at time $n+1$ (i.e.\ the set of edges $e$ such that $\bar{W}_e(n+1)\neq \bar{W}_e(n)$).
			
			Note that, by definition of $\tau_{n+1}$, $\gamma_{\tau_{n+1}}$ has the same distribution as $\gamma_{\tau_n+1}$ conditioned to intersect~$G_i$.
			Indeed, this can be justified by Lemma~\ref{lem:Zoe}: Assume for example that $i=1$ (the same argument holds for $i\in \{2, 3\}$). 
			Between time $\tau_n$ and time $\tau_{n+1}-1$, the edge-weights in $G_1$ stay constant, and some edge-weights in $G_2\cup G_3$ change. 
			By Lemma~\ref{lem:Zoe}, $\gamma_{\tau_{n+1}}$ is equal in distribution to $\mathrm{LE}(S)$, where $S$ is a random walk in $G_1$, started at $\nest_1$, stopped when first hitting $\food$, conditioned to hit $\food$ before returning to $\nest_1$, and whose transition probability are given by the edge-weights in $G_1$ at time $\tau_{n+1}$.
			Also by Lemma~\ref{lem:Zoe}, $\gamma_{\tau_n +1}$ conditioned to intersect $G_1$ is equal in distribution to $\mathrm{LE}(S')$, where $S'$ is a random walk in $G_1$, started at $\nest_1$, stopped when first hitting $\food$, conditioned to hit $\food$ before returning $\nest_1$, and whose transition probability are given by the edge-weights in $G_1$ at time $\tau_{n}+1$.
			Because the edge-weights in $G_1$ are the same at time $\tau_n+1$ and at time $\tau_{n+1}$,
			we do get that, in distribution, $S = S'$ and thus that $\gamma_{\tau_{n+1}}$ is equal in distribution to $\gamma_{\tau_n+1}$ conditioned to intersect $G_1$.

			By symmetry, we can assume without loss of generality that the $(\tau_n+1)$-th ant starts at $\nest_1$.
			We let $\Gamma_{\tau_n+1}$ be the trajectory of the $(\tau_n+1)$-th ant, so that ${\gamma_{\tau_n+1}} = \mathrm{LE}(\Gamma_{\tau_n+1})$; we also let $S_{\tau_n+1}$ be the part of the path $\Gamma_{\tau_n+1}$ after its last visit at $\nest_1$.
			By Lemma~\ref{lem:Zoe}, in distribution, $\mathrm{LE}(\Gamma_{\tau_n+1}) = \mathrm{LE}(S_{\tau_n+1})$.

			$\bullet$ First assume that $i=1$.
			Given that $S_{\tau_n+1}$ is in $G_1$, 
			then it is distributed as a random walk on $G_1$ 
			with transition probabilities proportional to the edge-weights at time $\tau_n$, 
			started at $\nest_1$, stopped when first hitting $\food$, 
			and conditioned not to return to $\nest_1$ before hitting $\food$.
			In other words, if we let $\Gamma_n^{\sss (1)}$ be a random walk on $G_1$ 
			with transition probabilities proportional to the edge-weights at time $\tau_n$, 
			started at $\nest_1$, stopped when first hitting $\food$, and let $S_n^{\sss (1)}$ be the part of $\Gamma_n^{\sss (1)}$ after its last visit to $\nest_1$, then, the distribution of $S_{\tau_n+1}$ conditioned on being in $G_1$  is the same as the distribution of $S_n^{\sss (1)}$.
			Thus, for any simple path $\gamma$ from $\nest_1$ to $\food$ in $G_1$, 
			\begin{align*}
				\mathbb P(\mathrm{LE}(\Gamma_{\tau_n+1})=\gamma | \mathrm{LE}(\Gamma_{\tau_n+1})\in G_1)
				&= \mathbb P(\mathrm{LE}(S_{\tau_n+1})=\gamma | \mathrm{LE}(S_{\tau_n+1})\in G_1)
				= \mathbb P(\mathrm{LE}(S_n^{\sss (1)}) = \gamma)\\
				&= \mathbb P(\mathrm{LE}(\Gamma_n^{\sss (1)}) = \gamma)
				= \mathbb P(\eta_{n+1} = \gamma),
			\end{align*}
			as desired.
			
			$\bullet$ We now assume that $i\in \{2, 3\}$.
			As in the case $i=1$, we have that the distribution of $S_{\tau_n+1}$ conditioned on being in $G_2\cup G_3$ is the same as $S^{\sss (23)}_n$, 
			where $S_n^{\sss (23)}$ is defined as the trajectory after its last visit at $\nest_1$ of a random walk $\Gamma_n^{\sss (23)}$ on $G_2\cup G_3$, 
			whose transition probabilities are proportional to the edges-weights at time $\tau_n$, started at $\nest_1$, and stopped when first hitting $\food$. 
			Reasoning as in the case $i=1$, 
			we get that, for any simple path $\gamma_3$ from $\nest_1$ to $\nest_2$ in $G_3$, 
			and for any simple path from $\nest_2$ to $\food$ in $G_2$,
			\[\mathbb P(\mathrm{LE}(\Gamma_{\tau_n+1}) = \gamma_3\gamma_2 | \mathrm{LE}(\Gamma_{\tau_n+1})\in G_2\cup G_3)
			= \mathbb P(\mathrm{LE}(\Gamma_n^{\sss (23)}) = \gamma_3\gamma_2).
			\]
			We now let $\Gamma_n^{\sss (3)}$ be the trajectory of $\Gamma_n^{\sss (23)}$ until it first hits $\nest_2$, and $\Gamma_n^{\sss (2)}$ be the rest of the trajectory, in which we have removed all excursions into $G_3$. With this notation, $\mathrm{LE}(\Gamma_n^{\sss (23)}) = \mathrm{LE}(\Gamma_n^{\sss (3)})\mathrm{LE}(\Gamma_n^{\sss (2)})$. Also, $\Gamma_n^{\sss (2)}$ and $\Gamma_n^{\sss (3)}$ are independent, which gives
			\[\mathbb P(\mathrm{LE}(\Gamma_{\tau_n+1}) = \gamma_1\gamma_2 | \mathrm{LE}(\Gamma_{\tau_n+1})\in G_2\cup G_3)
			= \mathbb P(\mathrm{LE}(\Gamma_n^{\sss (3)}) = \gamma_3)
			\mathbb P(\mathrm{LE}(\Gamma_n^{\sss (2)}) = \gamma_2).
			\]
			By definition, $\Gamma_n^{\sss (2)}$ is a random walk on $G_2$, started at $\nest_2$, stopped when first hitting $\food$ and with transition probabilities proportional to the edge-weights at time $\tau_n$. 
			Thus, $\mathbb P(\mathrm{LE}(\Gamma_n^{\sss (2)}) = \gamma_2) 
			= \mathbb P(\eta^{\sss (2)}_{n+1} = \gamma_2)$.
			If $\nest_1$ is the source of $G_3$ and $\nest_2$ its sink, then we also get that, straightforwardly,
			$\mathbb P(\mathrm{LE}(\Gamma_n^{\sss (3)}) = \gamma_3) 
			= \mathbb P(\eta^{\sss (3)}_{n+1} = \gamma_3)$.
			If $\nest_1$ is the sink of $G_3$ and $\nest_2$ its source, then the same conclusion follows by Lemma~\ref{claim:LEforward_egal_LEbackward}.
			In total, we thus get that
			\[\mathbb P(\mathrm{LE}(\Gamma_{\tau_n+1}) = \gamma_3\gamma_2 | \mathrm{LE}(\Gamma_{\tau_n+1})\in G_2\cup G_3)
			= \mathbb P(\eta^{\sss (2)}_{n+1} = \gamma_2)
			\mathbb P(\eta^{\sss (3)}_{n+1} = \gamma_3).\]
			Now note that, by definition, 
			\[\{\mathrm{LE}(\Gamma_{\tau_n+1})\in G_2\cup G_3\}
			=\{\mathrm{LE}(\Gamma_{\tau_n+1})\text{ intersects } G_2\}
			= \{\mathrm{LE}(\Gamma_{\tau_n+1})\text{ intersects } G_3\}.\]
			Therefore,
			\[\mathbb P(\mathrm{LE}(\Gamma_{\tau_n+1}) = \gamma_3\gamma_2 | \mathrm{LE}(\Gamma_{\tau_n+1})\text{ intersects } G_2)
			=  \mathbb P(\eta^{\sss (2)}_{n+1} = \gamma_2)
			\mathbb P(\eta^{\sss (3)}_{n+1} = \gamma_3).\]
			Summing over all simple paths $\gamma_3$ from $\nest_1$ to $\nest_2$ in $G_3$, 
			we get
			\[\mathbb P(\mathrm{LE}(\Gamma_{\tau_n+1}|_{G_2}) = \gamma_2 | \mathrm{LE}(\Gamma_{\tau_n+1})\text{ intersects } G_2)
			=  \mathbb P(\eta^{\sss (2)}_{n+1} = \gamma_2),\]
			and similarly for $i=3$, as desired.
		\end{proof}

		\subsection{Proof of Lemma \ref{lemma:conv_rn}}\label{subsect:proof_lemmeaux_isastochapprox}
		Recall that $N_i(n)$ denotes the number of ants that have reinforced a path in $G_i$ up to time $n$ in the multi-nest ant process.
		\begin{lemma}\label{lemma:Ni_poly}
			For any $ i \in \{1,2,3\}$, there exists $\eps_i>0$ such that, almost surely for all $n$ large enough, 
			\begin{equation}
				N_i(n)\geq n^{\eps_i}. \label{eqn:bound_on_N_i}
			\end{equation}
		\end{lemma}

		\begin{proof}[Proof of Lemma \ref{lemma:Ni_poly}]
			The proof is done in two steps: we first couple the multi-nest ant process on $G$ with a single-nest ant process on a simpler graph, and then use this simpler single-nest process to prove the result.
			
			{\bf A single-nest-ant process on a simpler graph:} 
			Recall that $\ell_i$ is the size of a shortest path in $G_i$ (from its source to its sink). 
			For any $i\in\{2,3\}$, we define $\ell'_i$ as the length of a longest simple path in $G_i$ (from its source to its sink). We let $G'$ be the graph $G$ in which the subgraph $G_1$ has been replaced by~$G'_1$, a path of $\ell_1$ edges in series, and the subgraphs $G_2$ and $G_3$ have been replaced by a path of $\ell'_2$, resp.\ $\ell'_3$, edges in series (see illustration on Figure \ref{fig:illu_couplage}).

			\begin{figure}[t]\centering
				\includegraphics[scale=0.7,page=6]{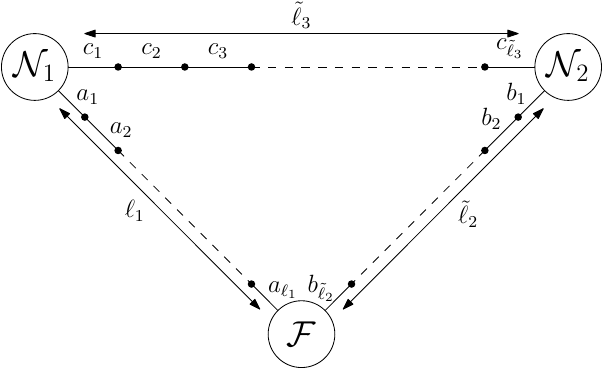}
				\caption{Illustration of $G'$.} \label{fig:illu_couplage}
			\end{figure}
			We let $(\bar{\bf W}^{\sss (i)}(n))_{n\geq 0}$ be the single-nest ant process on $G_i$, and $\bar C_{G_i}(n)$ be the effective conductance of $G_i$ with edge-weights $\bar{\bf W}^{\sss (i)}(n)$.
			By Lemma \ref{lem:bounds_KMS}, almost surely,
			\begin{equation}\label{eq:effcondG1}
				\bar C_{G_1}(n) \leq \frac{n+C_1}{\ell_1},
			\end{equation}
			and
			\begin{equation}\label{eq:effcondG23}
				\bar C_{G_2}(n)\geq \frac{n}{\ell'_2}\quad\text{ and }
				\quad \bar C_{G_3}(n)\geq \frac{n}{\ell'_2}.
			\end{equation}
			
			We now define a variant of the single-nest ant process on $G'$ as follows:
			we let
			\begin{align}
				\begin{cases}
					W'_{a_i}(0)  
					= C_1 & \forall1\leq i \leq \ell_1\\
					W'_{b_i}(0) 
					= 1 & \forall 1\leq i \leq \ell'_2 \\
					W'_{c_i}(0) 
					= 1 & \forall 1\leq i \leq \ell'_3,
				\end{cases}
			\end{align}
			where the edges $(a_i)_{1\leq i\leq \ell_1}$, $(b_i)_{1\leq i\leq \ell'_2}$ and $(c_i)_{1\leq i\leq \ell'_3}$ as in Figure~\ref{fig:illu_couplage}.
			Then, the weights $({\bf W}'(n))_{n\geq 1}$ are defined as in a single-nest ant process on $G'$ with nest $\nest_1$, source of food $\food$, and initial edge-weights ${\bf W}'(0)$, except that, when $\gamma_n= \{c_1, \ldots, c_{\ell'_3}, b_1, \ldots, b_{\ell'_2}\}$, then with probability $1-\alpha$, we instead set $\gamma_{n} = \varnothing$ (i.e.\ no edge-weight is increased at that step). 
			
			Note that, by definition, for all $n\geq 0$, for all $1\leq i\leq \ell'_1$, $1\leq j\leq \ell'_2$, and $1\leq k\leq \ell'_3$ 
			$N_1'(n):=W_{a_1}(n) = W_{a_i}(n)$,
			$N_2'(n):=W_{b_1}(n) =  W_{b_j}(n)$, and 
			$N_3'(n):=W_{c_1}(n) = W_{c_k}(n)$.
			Also note that $N_2'(n) = N'_3(n)$ for all $n\geq 0$.

			{\bf Coupling:}
			We now show that we can couple $({\bf W}'(n))_{n\geq 0}$ and $({\bf W}(n))_{n\geq 0}$ (the multi-nest ants process on $G$) so that, almost surely for all $n\geq 0$,
			\begin{align}
				N_1(n) + C_1 &\leq N'_1(n) \notag\\ 
				N_2(n) &\geq N'_2(n)\label{eqn:coupling}\\
				N_3(n) &\geq N'_3(n).\notag
			\end{align}
			The idea is that $({\bf W}'(n))_{n\geq 0}$ corresponds to the worst case scenario for the conductances of $G_2$ and $G_3$ in the multi-nest ant process: in this worst case scenario, whenever an ant reinforced $G_1$, it does so by choosing the shortest path in $G_1$, while whenever it reinforces $G_2$ or $G_3$, it does so by choosing the longest path in $G_2$ and $G_3$.
			
			For the coupling, we let $(B_n)_{n\geq 1}$ be a sequence of i.i.d.\ random variables such that, for all  $n\geq 1$, $\P(B_n = 1) = \alpha$ and $\P(B_n = 2) = 1-\alpha$.	
			We define the coupling by induction on $n\geq 0$: the base case is trivial by definition.
			For the induction step, we assume that \eqref{eqn:coupling} holds for some $n\geq 0$.
			
			$\bullet$ If $B_n = 1$, then the $(n+1)$-th ant in the multi-nest process on~$G$ starts from~$\nest_1$ (just like the $(n+1)$-th ant in the modified single-nest process on $G'$), and
			the ant process on $G'$ reinforces some edges (recall that in this modified version, there was a possibility to not reinforce any edge). 
			Note that, if we let $(\mathcal F_n)_{n\geq 0}$ be the natural filtration of $({\bf W}(n), {\bf W}'(n))_{n\geq 0}$, then
			\[\mathbb P(N'_1(n+1) = N'_1(n)+1 | \mathcal F_n) 
			= \frac{\frac{N'_1(n)}{\ell_1}}{\frac{N'_1(n)}{\ell_1} + \frac{N'_2(n)N'_3(n)}{\ell'_3 N'_2(n)+\ell'_2 N'_3(n)}}
			\geq \frac{\frac{N_1(n) + C_1}{\ell_1}}{\frac{N_1(n) + C_1}{\ell_1}+\frac{N_2(n)N_3(n)}{\ell'_3 N_2(n)+\ell'_2 N_3(n)}},\]
			by the induction hypothesis.
			Now note that, with the notation introduced in the proof of Proposition~\ref{prop:isastochapprox},
			\[\mathbb P(N_1(n+1) = N_1(n)+1 | \mathcal F_n)
			= \frac{C_{1}(n)}{C_{1}(n)+ \frac{C_{2}(n)C_{3}(n)}{C_{2}(n)+C_{3}(n)}}
			= \frac{\bar C_{G_1}(N_1(n))}{\bar C_{G_1}(N_1(n)) + \frac{\bar C_{G_2}(N_2(n))\bar C_{G_3}(N_3(n))}{\bar C_{G_2}(N_2(n))+\bar C_{G_3}(N_3(n))}},\]
			by Lemma~\ref{lemma:single_nest_in_multi_nest}.
			By~\eqref{eq:effcondG1} and~\eqref{eq:effcondG23}, we thus get
			\[\mathbb P(N_1(n+1) = N_1(n)+1 | \mathcal F_n)
			\leq \frac{\frac{N_1(n) + C_1}{\ell_1}}{\frac{N_1(n) + C_1}{\ell_1}+\frac{N_2(n)N_3(n)}{\ell'_3 N_2(n)+\ell'_2 N_3(n)}} \leq \mathbb P(N'_1(n+1) = N'_1(n)+1 | \mathcal F_n).\]
			We can thus couple the $(n+1)$-th ant in $G$ with the $(n+1)$-th ant in $G'$ so that
			if $N_1(n+1) = N_1(n)+1$ then $N'_1(n+1) =N'_1(n)+1$, and
			if $N'_2(n+1) = N_2(n)+1$ and $N'_3(n+1)= N_3(n)+1$ then $N_2(n+1)= N_2(n)+1$ and $N_3(n+1)= N_3(n)+1$. 
			This concludes the proof of the induction step in this case.		
			
			$\bullet$ If $B_n=2$, then the $(n+1)$-th ant in the multi-nest process on~$G$ starts from $\nest_2$ and, on $G'$, if the ant reaches $F$ through $G'_2$, then we set $N'_i(n+1) = N_i'(n)$ for all $i=1,2,3$, which implies in particular that, almost surely in this case
			\[N'_2(n+1) = N'_2(n)\quad\text{ and }\quad
			N'_3(n+1) = N'_3(n),\]
			which implies that the last two equations in~\eqref{eqn:coupling} hold at time $n+1$.
			Also, in this case,
			\begin{align*}
				\mathbb P(N_1(n+1) = N_1(n)+1|\mathcal F_n)
				&= \frac{\frac{C_{1}(n)C_{3}(n)}{C_{1}(n)C_{3}(n)}}
				{\frac{C_{1}(n)C_{3}(n)}{C_{1}(n)C_{3}(n)} + C_{2}(n)}
				\leq  \frac{C_{1}(n)}
				{C_{1}(n) + \frac{C_{2}(n)C_{3}(n)}{C_{2}(n)+C_{3}(n)}}\\
				&= \frac{\frac{C_1+N_1(n)}{\ell_1}}{\frac{C_1+N_1(n)}{\ell_1} + \frac{N_2(n)N_3(n)}{\ell'_3N_2(n)+\ell'_2N_3(n)}},
			\end{align*}
			by Lemma~\ref{lemma:single_nest_in_multi_nest}, \eqref{eq:effcondG1}, and~\eqref{eq:effcondG23}.
			Thus,
			\[\mathbb P(N'_1(n+1) = N'_1(n)+1|\mathcal F_n)
			= \frac{\frac{N'_1(n)}{\ell_1}}{\frac{N'_1(n)}{\ell_1}+ \frac{N'_2(n)N'_3(n)}{\ell'_3N'_2(n)+\ell'_2N_3(n)}}
			\geq \mathbb P(N_1(n+1) = N_1(n)+1|\mathcal F_n),\]
			by the induction hypothesis.
			Therefore, one can couple the $(n+1)$-th ant in $G'$ with the $(n+1)$-th ant in $G$ so that
			$N_1(n+1) = N_1(n)+1$ implies $N'_1(n+1)= N'_1(n)+1$, which concludes the proof of the induction step in this case.		
			
			{\bf Conclusion of the proof:}
			To conclude the proof, it is enough to show that there exists a constant $\eps>0$ such that, almost surely for all $n$ large enough, $N'_2(n) = N_3'(n) \geq n^{\eps}$.
			The advantage of the single-nest ants process on $G'$ is that $(N'_1(n), N'_2(n))$ 
			can be seen as a P\'olya urn
			with initial composition $(1,1)$, random replacement matrix
			\begin{equation}R =
				\begin{pmatrix}
					1 & 0\\
					0 & B(\alpha)
				\end{pmatrix} 
			\end{equation}
			where $B(\alpha)$ is a Bernoulli distribution of parameter $\alpha$,
			and weights $(1/\ell_1, 1/(\ell'_2+\ell'_3))$.
			Indeed, at every time step, we add one ball of colour~1 with probability proportional to $N'_1(n)/\ell_1$, or we add one ball of colour~2 with probability proportional to $\alpha N'_2(n)/(\ell'_2+\ell'_3)$.
			We refer the reader to~\cite{MR2562324} for a definition of P\'olya urns and apply Theorem~4 therein. We let $\mu_1 = 1/\ell_1$ and $\mu_{23} = \alpha/(\tilde{\ell_2}+\tilde{\ell_3})$. 
			By~\cite[Theorem~4]{MR2562324}, if $\mu_1 > \mu_{23}$, then $N'_1(n) = \mu_1 n + o(n)$ and $N_2'(n) = D n^{\mu_1/\mu_{23}}$, and otherwise, $N'_2(n) = D^* n + o(n)$ (where $D^*$ can be random if $\mu_1 = \mu_{23}$, but we do not care about this here). 
			In every case, this concludes the proof of Lemma \ref{lemma:Ni_poly} for $i\in\{2, 3\}$.
			By symmetry of $G_1$ and $G_2$ in $G$, we also get the same result for $i=1$.	
		\end{proof}
		
		We are now ready to prove Lemma \ref{lemma:conv_rn}.
		\begin{proof}[Proof of Lemma \ref{lemma:conv_rn}]
			Recall from the proof of Proposition~\ref{prop:isastochapprox} that, for all $i=1,2,3$,
			\[r_i(n) = \mathbb P(N_i (n + 1) = N_i (n) + 1|\mathcal F_n) - p_i(\hat{\bf N}(n))
			= P_i({\bf C}(n))- p_i(\hat{\bf N}(n)),\]
			with the functions $P_i$ and $p_i$ are defined in~\eqref{eqn:defP1}, \eqref{eqn:defP3}, and~\eqref{eqn:defp}.
			Equivalently, because $p(\hat{\bf N}(n)) = p({\bf N}(n))$,
			\[r(n)  = P({\bf C}(n)) - P\bigg(\frac{N_1(n)}{\ell_1}, \frac{N_2(n)}{\ell_2}, \frac{N_3(n)}{\ell_3}\bigg).\]
			For all $\delta>0$, for all $x = (x_1, x_2, x_3)\in (0,\infty)^3$ and $y = (y_1, y_2, y_3)\in (0,\infty)^3$ such that, for all $i=1,2,3$,
			\[(1-\delta)y_i\leq x_i\leq (1+\delta)y_i,\]
			we have
			\begin{align*}
				P_1(x)
				&= \alpha \frac{x_1}{x_1+\frac{x_2x_3}{x_2+x_3}}
				+(1-\alpha)  \frac{\frac{x_1x_3}{x_1+x_3}}{x_2+\frac{x_1x_3}{x_1+x_3}}\\
				&\leq \alpha \frac{(1+\delta)y_1}{(1+\delta)y_1+ (1-\delta)\frac{y_2y_3}{y_2+y_3}}
				+(1-\alpha)  \frac{(1+\delta)\frac{y_1y_3}{y_1+y_3}}{(1-\delta)y_2+(1+\delta)\frac{y_1y_3}{y_1+y_3}}
				\leq \frac{1+\delta}{1-\delta} P_1(y).
			\end{align*}
			The first inequality holds because $(a, b)\in (0,\infty)^2\mapsto a/(a+b)$ is an increasing function of~$a$ and a decreasing function of $b$, and $(a,b)\in(0,\infty)^2\mapsto {ab}/(a+b)$ is an increasing function of both~$a$ and~$b$.
			Similarly, one gets $P_1(x)\geq \frac{1-\delta}{1+\delta}P_1(y)$, and, similarly,
			\begin{equation}\label{eq:lip}
				\frac{1-\delta}{1+\delta}P_i(y)\leq P_i(x)\leq \frac{1+\delta}{1-\delta}P_i(y),
			\end{equation}
			for all $i=1,2,3$.
			
			By Lemma~\ref{lem:bounds_KMS}, there exist three random variables $K_1$, $K_2$, and $K_3$ such that, for all $i=1,2,3$, almost surely for all $n\geq 0$,
			\[\bar C_{G_i}(n)\leq \frac{n-K_i n^{\alpha_i}}{\ell_i},\]
			where we recall that $\bar C_{G_i}(n)$ is the effective conductance of $G_i$ (between its source and its nest) at time $n$ of the single-nest ant process run on $G_i$,
			and where we have set $\alpha_i = \ell_i/(\ell_i+1)$.
			This implies that, for all $i=1,2,3$,
			\begin{equation}
				\frac{n}{\ell_i}\left(1 - \ell_i K_i n^{\alpha_i-1}\right) \leq 
				\bar C_{G_i}(n)  \leq 
				\frac{n}{\ell_i}\left(1 + \ell_i K_i n^{\alpha_i-1}\right).
			\end{equation}
			By Lemma~\ref{lemma:Ni_poly}, for all $i=1,2,3$, almost surely for all $n$ large enough, $N_i(n) \geq n^{\eps_i}$. 
			Thus, if we let $\Delta_n \coloneqq \sup_i \{\ell_i K_i n^{(\alpha_i-1)\eps_i}\}$, then, almost surely for all $n$ large enough,
			\begin{equation}
				\frac{N_i(n)}{\ell_i}\left(1 - \Delta_n\right)  
				\leq	\bar C_{G_i}(N_i(n)) = C_{i}(n)
				\leq 	\frac{N_i(n)}{\ell_i}\left(1 + \Delta_n\right),
				\label{eqn:bound_conductance_delta_n}
			\end{equation} 
			where the middle equality holds by Lemma~\ref{lemma:single_nest_in_multi_nest}.
			Thus,~\eqref{eq:lip} applies to $\delta = \Delta_n$, $x = {\bf C}(n)$, and $y = (N_1(n)/\ell_1, N_2(n)/\ell_2, N_3(n)/\ell_3)$ , which gives
			\begin{equation}\label{eq:use_later}
				\bigg(\frac{1-\Delta_n}{1+\Delta_n}-1\bigg) p_i({\bf N}(n))
				\leq r_i(n)\leq 
				\bigg(\frac{1+\Delta_n}{1-\Delta_n}-1\bigg) p_i({\bf N}(n)).
			\end{equation}
			Because, by definition, $p_i$ takes values in $[0,1]$, we get that
			\[\|r(n)\|_{\infty} \leq \mathcal O(\Delta_n),\]
			almost surely as $n\uparrow\infty$, because $\Delta_n\to0$ almost surely, by definition.
			By definition, $\Delta_n = \mathcal O(n^{-\varepsilon})$ for some $\varepsilon>0$. Thus,
			$\sum_{n\geq 1} \|r(n)\|/n<\infty$ as required.
		\end{proof}
		
		The following lemma, which follows straightforwardly from~\eqref{eq:use_later}, will be used later in the proof of Theorem~\ref{thm:trianglesSP}:
		\begin{lemma}\label{cor:bound_P_p} 
			Let $i\in\{1, 3\}$. Almost surely for all $n$ large enough,
			\[\frac{1-\Delta_n}{1+\Delta_n}\cdot p_i(\hat{\bf N}(n)) 
			\leq \mathbb P\big(N_i(n+1) = N_i(n)+1 | \hat {\bf N}(n)\big)
			\leq \frac{1+\Delta_n}{1-\Delta_n}\cdot p_i(\hat{\bf N}(n)),
			\]
			where $\Delta_n =  \sup_i \{ \ell_i K_i n^{(\alpha_i-1)\eps_i}\}$.
		\end{lemma}
		
		\begin{remark}
			When every $G_i$ is a simple path, then $\Delta_n=0$, and it implies that $r(n) = 0$. 
			Intuitively, the fact that $\sum_n n^{-1}\| r(n)\| <\infty$ indicates that, asymptotically, the model behaves as if $G_1, G_2$, and $G_3$ were all replaced by simple paths with respectively $\hmin(G_1)$, $\hmin(G_2)$, and $\hmin(G_3)$ edges in series.
		\end{remark}

		\section{Proof of the main result (Theorem \ref{thm:trianglesSP})}
		
		The aim of this section is to prove Theorem \ref{thm:trianglesSP}.
		Without loss of generality, we assume throughout that $\ell_1 \leq \ell_2$.
		In Section~\ref{sect:stratandpremresults}, we have proved that 
		the multi-nest ants process restricted to $G_1$ (resp.\ $G_2$, resp.\ $G_3$) behaves like a single-nest ants process (see Proposition~\ref{lemma:single_nest_in_multi_nest}).
		Therefore, by the results of~\cite{kious2020finding} on the single-nest ants process on series-parallel graphs, 
		Theorem~\ref{thm:trianglesSP} follows from the convergence of $\hat{\bf N}(n)$.
		
		To prove convergence of $\hat{\bf N}(n)$ to the limits announced in Theorem~\ref{thm:trianglesSP}, we use the fact that, as proved in Proposition~\ref{prop:isastochapprox}, $(\hat{\bf N}(n))_{n\geq 0}$ is a stochastic approximation with vector-field~$F$. Heuristically, this means that $(\hat {\bf N}(n))_{n\geq 0}$ asymptotically ``follows the flow'' of the ODE $\dot{y} = F(y)$.
		We use Bendixson–Dulac theorem to prove that any solution of the ODE $\dot{y} = F(y)$ converges to a stable zero of~$F$ (see Lemma \ref{lemma:triangles_noorbits}). 
		We then deduce that the normalized process $(\hat{\bf N}(n))_{n\geq 0}$ 
		also converges to a zero of $F$ (see Corollary \ref{cor:cv}).
		Then it remains to determine which of the zeros of $F$ the process converges to. 
		We state and prove a series of lemmas (Lemmas \ref{lemma:minw1}, \ref{lemma:minw3} and \ref{lemma:majw1}) that enable us to eliminate the unstable zeros one by one. 
		The main idea is to compare the process $(\hat {\bf N}(n))_{n\geq0}$ 
		with one-dimensional processes that we analyse using convergence results on generalized P\'olya urn processes given in Section~\ref{sect:GPurn}. 
		
		\begin{figure}[t]
			\centering\def\svgwidth{0.55\columnwidth}
			\includegraphics[width=0.6\textwidth]{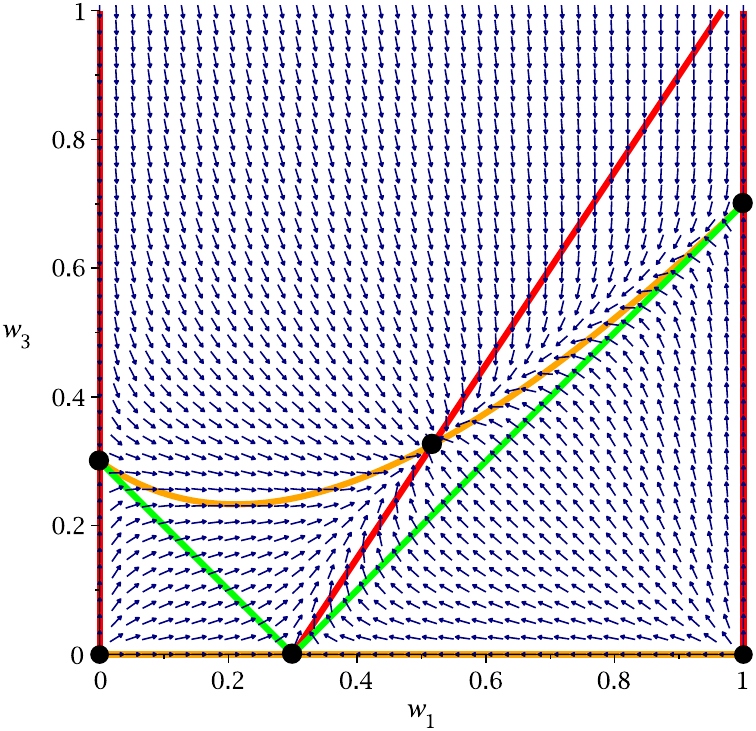}
			\caption{\small 
				Here $\ell_1 = 2$, $\ell_2 = 4$, $\ell_3 = 3$ and $\alpha = 0.3$. In particular, $\ell_2 < \ell_1+ \ell_3$ and $\ell_3 < \ell_1+ \ell_2$ (this case is of particular interest since $(\beta_1,\beta_3)\in (0,1)^2$).
				Blue arrows represent the vector field~$F$. 
				Orange curves represent the solutions to $F_3(w)=0$, while red ones are solutions to $F_1(w)=0$. 
				The black dots are thus the solutions to $F(w) = 0$; these zeros are $ (\alpha,0)$, $(0,\alpha)$, $(1,1-\alpha)$, $(0,0)$, $(1,0)$ and $(\beta_1,\beta_3)$, with $(\beta_1,\beta_3)\in[0,1]^2$ if and only if $\ell_2 \leq \ell_1+ \ell_3$ and $\ell_3 \leq \ell_1+ \ell_2$, see Lemma~\ref{lem:zerosofF}. 
				The two green lines represent the solutions to $w_3 + w_1 = \alpha$ and $w_3 - w_1 = -\alpha$; by Lemma~\ref{lemma:inequalities_for_the_nests}, $\hat{\bf N}(n)$ will almost surely, eventually, be above this green line.}
			\label{fig:ideegenerale}
		\end{figure}

		We provide in Figures~\ref{fig:ideegenerale} and~\ref{fig:imagesMapleautrescas} some pictures of the vector field $F$ for some particular values of $\alpha$, $\ell_1$, $\ell_2$ and $\ell_3$. These give an idea of what the flow of the ODE $\dot y = F(y)$ looks like and to which zero of $F$ the stochastic approximation ``should'' converge to.

		\begin{figure}[htbp]
			\centering
			\begin{subfigure}[b]{0.45\textwidth}
				\centering\def\svgwidth{\columnwidth}
				\includegraphics[page=1,width=0.95\columnwidth]{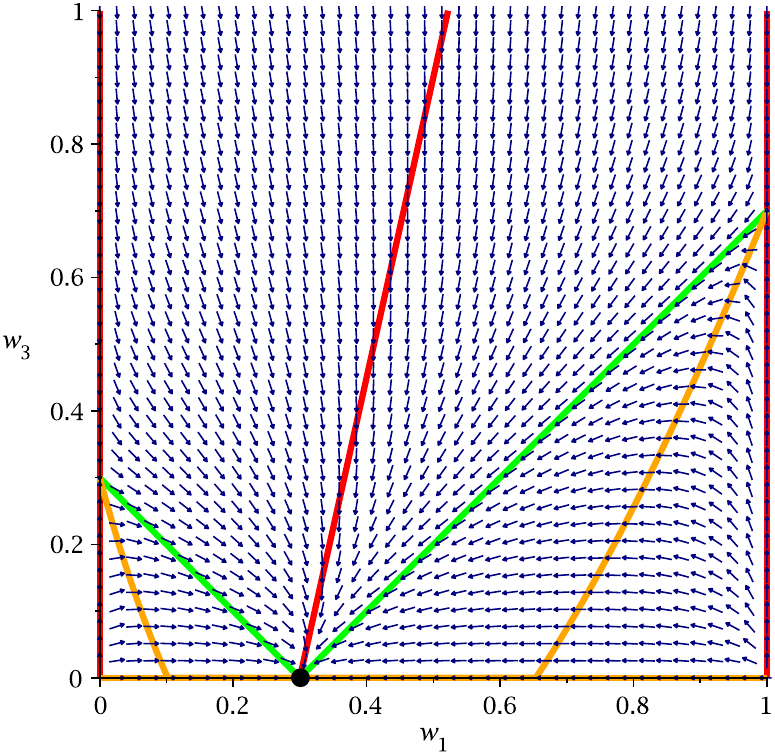}
				\hfill
				\subcaption{Here $\ell_1 = 2$, $\ell_2 = 4$, $\ell_3 = 9$ and $\alpha = 0.3$. This figure illustrates what happens when $\ell_3 > \ell_1+ \ell_2$.}
				\label{subfig:b}
			\end{subfigure}
			\hfill
			\begin{subfigure}[b]{0.45\textwidth}
				\centering\def\svgwidth{\columnwidth}
				\includegraphics[page=1,width=0.95\columnwidth]{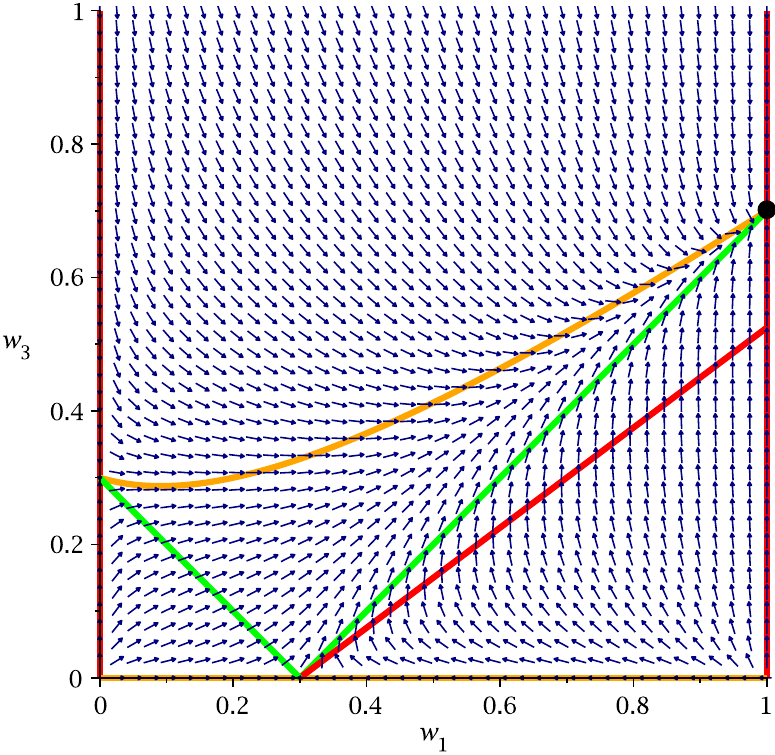}
				\hfill
				\subcaption{Here $\ell_1 = 2$, $\ell_2 = 6$, $\ell_3 = 3$ and $\alpha = 0.3$. This figure illustrates what happens when $\ell_2 > \ell_1+ \ell_3$.}
				\label{subfig:c}
			\end{subfigure} \hfill 
			\begin{subfigure}[b]{0.45\textwidth}
				\centering\def\svgwidth{\columnwidth}
				\includegraphics[page=1,width=0.95\columnwidth]{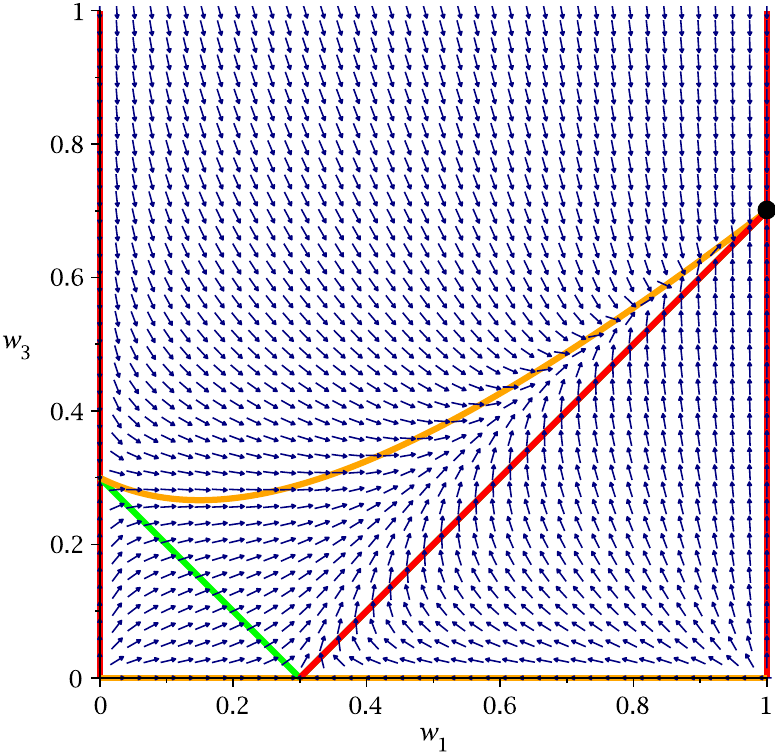}
				\hfill
				\subcaption{Here $\ell_1 = 2$, $\ell_2 = 5$, $\ell_3 = 3$ and $\alpha = 0.3$. This figure illustrates what happens when $\ell_2= \ell_1+ \ell_3$ (i.e.\ $(\beta_1,\beta_3)=(1,1-\alpha)$).}
				\label{subfig:d}
			\end{subfigure} \hfill
			\begin{subfigure}[b]{0.45\textwidth}
				\centering\def\svgwidth{\columnwidth}
				\includegraphics[page=1,width=0.95\columnwidth]{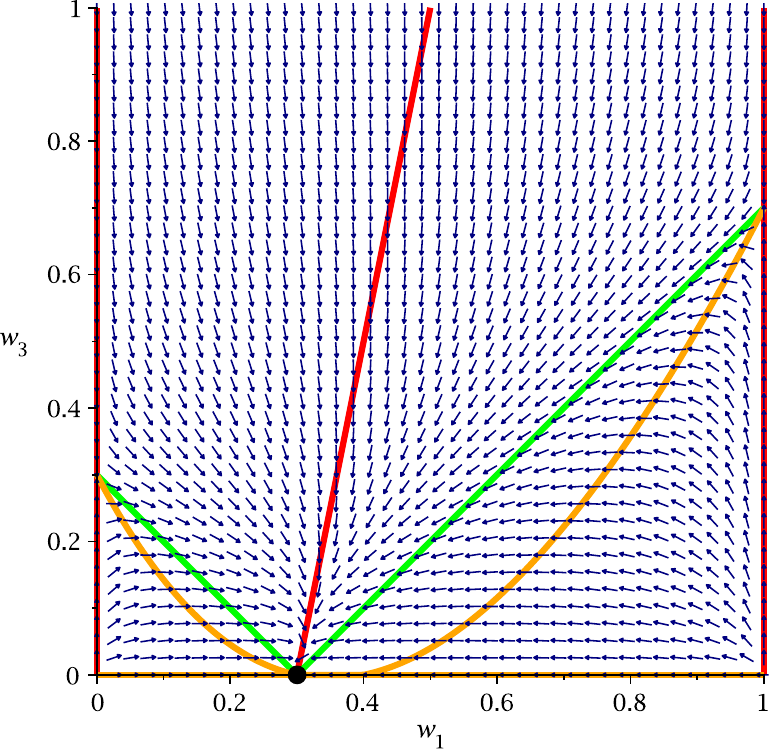}
				\hfill
				\subcaption{Here $\ell_1 = 2$, $\ell_2 = 3$, $\ell_3 = 5$ and $\alpha = 0.3$. This figure illustrates what happens when $\ell_3 = \ell_1+ \ell_2$ (i.e.\ $(\beta_1,\beta_3)=(\alpha,0)$).}
				\label{subfig:e}
			\end{subfigure}
			\caption{\small This figure represents the same objects as Figure \ref{fig:ideegenerale} (see its caption for a detailed explanation), but in different cases that highlight the distinct behaviours that can occur, depending on the parameters $\ell_1$, $\ell_2$ and $\ell_3$. Here we have left only one black dot on every figure, which is the almost sure limit of $\hat{\bf N}(n)$.}
			\label{fig:imagesMapleautrescas}
		\end{figure}

	\subsection{The zeros of the vector field~$F$}
	Before starting the proof, recall that, by definition, for all $n\geq 0$, $\hat{N}_1(n) + \hat{N}_2(n) = 1$. Thus, instead of considering the $3$-dimensional vector $\hat{\bf N}(n) = (N_1(n), N_2(n), N_3(n))/n$, it is enough to consider the two-dimensional vector $(N_1(n), N_3(n))/n$, which we also let $\hat{\bf N}(n)$ denote, by a slight abuse of notation.
	This implies that, instead of considering $F$ as a vector field on $\mathbb R^3$, we consider instead the vector field $(w_1, w_3)\mapsto (F_1(w_1, 1-w_1, w_3), F_3(w_1, 1-w_1, w_3))$, which, by a slight abuse of notation, we also let $F$ denote.
	\begin{lemma}\label{lem:zerosofF}
		For all $(w_1, w_2, w_3)\in \mathcal E$, 
		$F(w_1, w_2, w_3) = 0$ if and only if
		\[(w_1,w_3) \in \left\{(0,0), (0,\alpha), (\alpha,0), (\beta_1,\beta_3), (1,1-\alpha), (1,0)\right\},\]
		where $(\beta_1, \beta_3)$ is as in Theorem~\ref{thm:trianglesSP}.
	\end{lemma}
	
	\begin{proof}
		By definition, for all $w\in\mathcal E$,
		\begin{align}
			F_1(w) 
			&= \alpha \cdot \frac{\frac{w_1}{\ell_1}}{\frac{w_1}{\ell_1}+\frac{w_2w_3}{\ell_3w_2+\ell_2w_3}}
			+ (1-\alpha)\cdot \frac{\frac{w_1w_3}{\ell_3w_1+\ell_1w_3}}{\frac{w_2}{\ell_2} + \frac{w_1w_3}{\ell_3w_1+\ell_1w_3}}-w_1
			\notag\\
			&= \frac{\alpha w_1(\ell_3w_2+\ell_2w_3)}{w_1(\ell_3w_2+\ell_2w_3) + \ell_1w_2w_3}
			+ \frac{(1-\alpha)\ell_2w_1w_3}{(\ell_3w_1+\ell_1w_3)w_2 + \ell_2w_1w_3}-w_1
			\notag\\
			&= w_1\cdot \frac{\alpha\ell_3 w_2 + \ell_2w_3 - \ell_3w_1w_2-\ell_2 w_1w_3-\ell_1w_2w_3}{\ell_3w_1w_2+\ell_2 w_1w_3+\ell_1w_2w_3}
			\notag\\
			&=\frac{w_1 (1-w_1) \left( \ell_3 (\alpha -w_1) + (\ell_2-\ell_1) w_3 \right)}{w_3(\ell_1 + w_1(\ell_2 - \ell_1)) + \ell_3 w_1 (1-w_1)}.
			\label{eq:F1}
		\end{align}
		In the last equality, we have used the fact that $w_2 = 1-w_1$ for all $w\in\mathcal E$.
		Thus, $F_1(w) = 0$ if and only if 
		\begin{equation}\label{eq:def_gamma}
			w_1 = 0 \text{ \ or \ } w_1 = 1 \text{ \ or \ }
			w_3 = \frac{\ell_3(w_1-\alpha)}{\ell_2-\ell_1} \eqqcolon \gamma(w_1).
		\end{equation}
		Similarly,
		\begin{equation}\label{eq:F3}
			F_3(w) = \frac{w_3 \left(\ell_2(1-\alpha)w_1 + \ell_1 \alpha (1-w_1) - \ell_3 w_1 (1-w_1) - w_3 (\ell_1 + w_1(\ell_2-\ell_1))\right)}{w_3(\ell_1 + w_1(\ell_2 - \ell_1)) + \ell_3 w_1 (1-w_1)},
		\end{equation}
		and thus, $F_3(w) = 0$ if and only if 
		\begin{equation}\label{eq:def_g}
			w_3 = 0 \text{ \ or \ } 
			w_3 = \frac{\ell_2(1-\alpha)w_1 + \ell_1 \alpha (1-w_1) - \ell_3 w_1 (1-w_1)}{\ell_1 + w_1(\ell_2-\ell_1)} \eqqcolon g(w_1).
		\end{equation}
		Finally, $g(w_1) = \gamma(w_1)$ if and only if $w_1 = \beta_1$, and $\gamma(\beta_1) = \beta_3$.
		
		Combining (\ref{eq:def_gamma}) and (\ref{eq:def_g}) allows us to conclude: 
		Indeed, we first note that $w_1=\alpha$ is the only solution to $\gamma(w_1) = 0$. 
		Then, in the table below, any solution of $F(w) = 0$ must satisfy one of the equations in the left-most column and one of the equations in the top line: for each combination, there is a unique solution, which we write in the corresponding cell:
		\begin{equation*}
			\begin{array}{ |c|c|c|c| } 
				\hline
				& w_1 = 0 		& w_1 = 1 	& w_3 = \gamma(w_1)\\ \hline
				w_3=0			& (0,0) 		& (1,0) 	& (\alpha,0) \\ \hline
				w_3 = g(w_1)	& (0,\alpha)	& (1,1-\alpha)		& (\beta_1,\beta_3) \\ 
				\hline
			\end{array}
		\end{equation*}
		For example, $w =  (\beta_1, \beta_3)$ is the unique solution of $w_3 = g(w_1)=\gamma(w_1)$.
	\end{proof}

	\begin{remark}\label{rem:beta1andbeta3}
		Note that that $(\beta_1,\beta_3) \in [0,1]^2$ if and only if $\ell_2 \leq \ell_1+\ell_3$ and $\ell_3 \leq \ell_1+\ell_2$. 
		Indeed, a straightforward computation gives
		\begin{equation}
			\beta_1 < 1 \ \iff \ \ell_2 < \ell_1 + \ell_3 \ \et \ \beta_3 > 0 \ \iff \ \ell_3 < \ell_1 + \ell_2  \label{eqn:conditionsonBeta}
		\end{equation}
		(and also $\beta_1 = 1 \iff \ell_2 = \ell_1 + \ell_3$ and $\beta_3 = 0 \iff \ell_3 = \ell_1 + \ell_2$).
		Moreover, since $\beta_3 = \gamma(\beta_1)$, $\beta_3\geq 0$ implies directly that $\beta_1 \geq \alpha >0$. 
		Then, using the facts that $\alpha\in(0,1)$ and, 
		for any $w_1\in[0,1]$, $w_1(1-w_1)\geq 0$, we get 
		\begin{equation*}
			\ell_2(1-\alpha)w_1 + \ell_1 \alpha (1-w_1) - \ell_3 w_1 (1-w_1) < \ell_2 w_1 + \ell_1(1-w_1),
		\end{equation*}
		i.e.\ $g(w_1)< 1$, which implies that $\beta_3 = g(\beta_1) < 1$.
	\end{remark}

	\subsection{Flow of the ODE}
	\begin{lemma}\label{lemma:triangles_noorbits}
		For all $w\in [0,1]^2$, we let $t\in[0,\infty)\mapsto y_w(t)$ be the solution of $\dot{y} = F(y)$ such that $y_w(0) = w$. 
		For all $w\in [0,1]^2$, $\lim_{t\to\infty} y_w(t)$ exists.
	\end{lemma}

	\begin{proof}
		The idea of this proof is to first use the Bendixson–Dulac theorem (see for example \cite[Section 4.1]{hofbauer1998evolutionary}) to show that the ODE has no non-constant periodic solution (which we call orbit). 
		And then use the Poincaré-Bendixson theorem (see, for example \cite[Theorem 4.1.1]{hofbauer1998evolutionary}) to conclude the proof.
		
		{The Bendixson-Dulac theorem} states that, if there exists a continuously-differentiable function $g : [0,1]^2\to \mathbb R$ such that, for all $w\in[0,1]^2$ (except, possibly, in a set of measure~0),
		\[\frac{\partial(gF_1)}{\partial w_1}(w) + \frac{\partial(gF_2)}{\partial w_2}(w)>0,\]
		then the ODE $\dot y = F(y)$ admits no orbit.
		
		We let 
		\[g(w) = -\frac{w_3(\ell_1 + w_1(\ell_2 - \ell_1)) + \ell_3 w_1 (1-w_1)}{w_1w_3(1-w_1)}.\]
		for all $w\in(0,1)\times(0,1]$, and $g(w) = 0$ otherwise. 
		By~\eqref{eq:F1} and~\eqref{eq:F3},
		\begin{align*}
			-(gF_1)(w)
			&= \frac{\ell_3(\alpha - w_1)}{w_3} +  \ell_2 - \ell_1, \quad\text{ and }\\
			-(gF_3)(w)
			&= \frac{\ell_2(1-\alpha)-(\ell_2-\ell_1)w_3}{1-w_1} + \frac{\ell_1\alpha}{w_1} - \ell_3 - \frac{\ell_1w_3}{w_1(1-w_1)},
		\end{align*}
		which imply
		\begin{align*}
			\frac{\partial (g F_1)}{\partial w_1}(w) = \frac{\ell_3}{w_3}
			\quad\text{ and }\quad
			\frac{\partial (gF_3)}{\partial w_3} (w)
			= \frac{\ell_2-\ell_1}{1-w_1} + \frac{\ell_1}{w_1(1-w_1)}.
		\end{align*}
		Thus, for all $w\in(0,1)\times(0,1]$,
		\begin{align*}
			\frac{\partial (gF_1)}{\partial w_1}(w)+\frac{\partial (gF_3)}{\partial w_3}(w)
			&=  \frac{\ell_3}{w_3} + \frac{\ell_2-\ell_1}{1-w_1} + \frac{\ell_1}{w_1(1-w_1)}
			= -g(w) >0.
		\end{align*} 
		Therefore, the Bendixson-Dulac theorem (see, e.g., \cite[Section 4.1]{hofbauer1998evolutionary}) applies and gives that every solution of $\dot y = F(y)$ starting in $(0,1)\times (0,1]$ has no orbit. 
		
		The Poincar\'e-Bendixson theorem (see, e.g., \cite[Theorem 4.1.1]{hofbauer1998evolutionary}) states that, because $y_w$ takes values in $[0,1]^2$, the limit set of $y_w$ is either a zero of $F$ or an orbit of $F$. Because there are no orbits, this concludes the proof. 
	\end{proof}
	
	We thus have the following corollary, which states that the normalised process $(\hat{\bf N}(n))_{n\geq 0}$ converges almost surely as $n$ tends to infinity:
	\begin{corollary}\label{cor:cv}
		Almost surely, $\lim_{n \to \infty}{\hat{\bf N}}(n)$ exists and 
		\[\lim_{n \to \infty}{\hat{\bf N}}(n) \in \{ (1,1-\alpha),(0, \alpha), (\alpha, 0), (\beta_1,\beta_3)\}.\]
	\end{corollary}
	
	\begin{proof}
		By Proposition \ref{prop:isastochapprox}, 
		$(\hat{N}(n))_{n\geq 0}$ is a stochastic approximation with vector field $F$.
		By Lemma~\ref{lemma:triangles_noorbits}, any solution to the ODE $\dot y = F(y)$ started on $[0,1]^2$ converges. 
		Thus, Proposition~\ref{prop:ODEandStochApproxCV} to $(\hat{N}(n))_{n\geq 0}$ and gives that, almost surely,
		\[\lim_{n\to\infty} \hat {\bf N}(n) \in \{w: {F}(w) = {0}\}
		=\{ (1, 1-\alpha),(0, \alpha), (\alpha,0), (0,0), (1,0), (\beta_1,\beta_3)\},\] 
		by Lemma~\ref{lem:zerosofF}.
		Lemma \ref{lemma:inequalities_for_the_nests} gives that, almost surely,
		$\lim_{n\to\infty} \hat {\bf N}(n)\notin \{(0,0), (1,0)\}$, which concludes the proof.
	\end{proof}
	
	\subsection{Unstable zeros}
	From Corollary~\ref{cor:cv}, we have four possible values for $\lim_{n\to\infty} \hat{\bf N}(n)$. 
	In this section, we show almost sure lower and upper bounds for $\lim_{n\to\infty} \hat{\bf N}(n)$, 
	which then allow us to discard three out of the four possible limits and thus prove Theorem~\ref{thm:trianglesSP}.
	
	\begin{lemma}\label{lemma:minw1}
		Almost surely, $\liminf_{n\to\infty} \hat{N}_1(n) \geq \alpha$.
	\end{lemma}
	
	\begin{proof}
		Recall that, by Lemma~\ref{cor:bound_P_p}, almost surely for all $n$ large enough,
		\[\frac{1-\Delta_n}{1+\Delta_n}\cdot p_1(\hat{\bf N}(n)) 
		\leq \mathbb P\big(N_1(n+1) = N_1(n)+1 | \hat {\bf N}(n)\big)
		\leq \frac{1+\Delta_n}{1-\Delta_n}\cdot p_1(\hat{\bf N}(n)),
		\]
		where $\Delta_n =  \sup_i \{ \ell_i K_i n^{(\alpha_i-1)\eps_i}\}$, $K_1, K_2, K_3$ are almost surely finite random variables, $\eps_1, \eps_2, \eps_3$ are three positive constants, and $\alpha_i = \ell_i/(\ell_i+1)$ for all $1\leq i\leq 3$.
		This implies, in particular, that, for all $\varepsilon>0$, 
		almost surely for all $n$ large enough,
		\begin{equation}\label{eq:truc}
			\mathbb P\big(N_1(n+1) = N_1(n)+1 | \hat {\bf N}(n)\big)
			\geq \frac{1-\Delta_n}{1+\Delta_n}\cdot p_1(\hat{\bf N}(n))
			\geq (1-\varepsilon) g(\hat N_1(n)),
		\end{equation}
		where we have set, for all $w_1\in[0,1]$,
		\[g(w_1) = \inf_{0\leq w_3\leq 1} p_1(w_1, 1-w_1, w_3).\]
		Our aim is to apply Corollary~\ref{cor:urnstochdomin} to the case $G = (1-\varepsilon)g$.
		To do so, we need to lower bound $g$ by a linear function of slope at least~1, in a sufficiently large neighbourhood of $0$.
		Recall that, by definition (see~\eqref{eqn:defP1} and~\eqref{eqn:defp}),
		\[p_1(w) =  \frac{\alpha \ell_3 w_1(1-w_1)+\ell_2 w_1w_3}{\ell_3 w_1(1-w_1)+ \ell_2 w_1 w_3 + \ell_1 (1-w_1)w_3}.\]
		This implies
		\[\frac{\partial p_1}{\partial w_3}(w)
		= \frac{\ell_3 w_1 (1-w_1) ((1-\alpha) \ell_2 w_1 - \alpha\ell_1  (1-w_1)))}{(\ell_3 w_1(1-w_1)+ \ell_2 w_1 w_3 + \ell_1 (1-w_1)w_3)^2}.\]
		Thus, $w_3\mapsto p_1(w)$ is a non-decreasing function if and only if 
		\[w_1 \geq \frac{\ell_1 \alpha}{\ell_2 (1-\alpha) + \ell_1 \alpha} \eqqcolon c_1.\]
		We thus get that
		\[g(w_1) = \begin{cases}
			p_1(w_1, 1-w_1, 0) = \alpha & \si w_1 \geq c_1,\\
			p_1(w_1, 1-w_1, 1) = \frac{w_1 \left(\alpha \ell_3 (1-w_1) + \ell_2\right)}{\ell_1(1-w_1) + \ell_2 w_1 + \ell_3 w_1 (1-w_1)} & \sinon.
		\end{cases}\]
		Note that $g$ is continuous since $w_3\mapsto p_1(c_1, 1-c_1, w_3)$ is constant and thus  $p_1(c_1,1-c_1,1) = p_1(c_1, 1-c_1, 0)=\alpha$.
		Also, because, by assumption, $\ell_2 \geq \ell_1$, we get that $\alpha \geq c_1$.
		
		We now show that, on $[0,\alpha]$, the function~$g$ only cross the diagonal at $w_1=0$ and $w_1 = \alpha$ (only if $c_1 = \alpha$). Indeed, for $w_1 <c_1$, $g(w_1) = w_1$ if and only if $w_1 = 0$ or 
		\begin{equation}\label{eq:sol?}
			\frac{\alpha \ell_3 (1-w_1) + \ell_2}{\ell_1(1-w_1) + \ell_2 w_1 + \ell_3 w_1 (1-w_1)} = 1.
		\end{equation}
		Note that~\eqref{eq:sol?} is equivalent to $(\alpha-w_1)\ell_3 = \ell_1-\ell_2$, which admits no solution on $[0,\alpha)$ because $\ell_1\leq \ell_2$, by assumption.
		Thus, as claimed, on $[0,\alpha]$, the function~$g$ only crosses the diagonal at $w_1=0$ and $w_1 = \alpha$.
		In fact, $g$ is (strictly) above the diagonal on $(0,\alpha)$, since its derivative at~$0$ is $(\alpha\ell_3+\ell_2)/\ell_1>1$ (since $\ell_2 \geq \ell_1$ and $\alpha \ell_3>0$).
		Therefore, for any $\eta>0$, there exists $\eps>0$ and $\delta >0$ small enough such that, for all $w_1 \in (0,\alpha-\eta)$, $(1-\eps)g(w_1) > (1+\delta) w_1$.
		By Corollary~\ref{cor:urnstochdomin}, this implies
		\[\liminf_{n\to\infty} \hat N_1(n) \geq \alpha-\eta,\]
		which concludes the proof, because this holds for all $\eta>0$.
	\end{proof}
		
		\begin{lemma}\label{lemma:minw3}
			If $ \ell_3< \ell_1 +\ell_2$, then $\liminf_{n\to\infty} \hat N_3(n)>0$ almost surely.
		\end{lemma} 
		
		\begin{proof}
			We fix $\delta>0$ small enough so that
			\[\frac{\alpha \ell_1 (1-\alpha-\delta)+(\alpha-\delta) (1-\alpha)\ell_2}{\ell_3 (\alpha+\delta)(1-\alpha+\delta)}>1.\]
			This is possible because
			\[\lim_{\delta\downarrow0}
			\frac{\alpha \ell_1 (1-\alpha-\delta)+(\alpha-\delta) (1-\alpha)\ell_2}{\ell_3 (\alpha+\delta)(1-\alpha+\delta)}
			=\frac{\ell_1+\ell_2}{\ell_3}>1,\]
			by assumption.
			We then fix $\varepsilon>0$ small enough so that
			\begin{equation}\label{eq:choice_eps}
				\frac{\alpha \ell_1 (1-\alpha-\delta)+(\alpha-\delta) (1-\alpha)\ell_2}{\ell_3 (\alpha+\delta)(1-\alpha+\delta)}>\frac{1+\varepsilon}{(1-\varepsilon)^2}.
			\end{equation}
			Recall that, by Lemma~\ref{cor:bound_P_p}, almost surely for all $n$ large enough,
			\[\frac{1-\Delta_n}{1+\Delta_n}\cdot p_1(\hat{\bf N}(n)) 
			\leq \mathbb P\big(N_1(n+1) = N_1(n)+1 | \hat {\bf N}(n)\big)
			\leq \frac{1+\Delta_n}{1-\Delta_n}\cdot p_1(\hat{\bf N}(n)),
			\]
			where $\Delta_n =  \sup_i \{ \ell_i K_i n^{(\alpha_i-1)\eps_i}\}$, $K_1, K_2, K_3$ are almost surely finite random variables, $\eps_1, \eps_2, \eps_3$ are three positive constants, and $\alpha_i = \ell_i/(\ell_i+1)$ for all $1\leq i\leq 3$.
			This implies, in particular, that,
			almost surely for all $n$ large enough,
			\begin{equation}\label{eq:truc3}
				\mathbb P\big(N_3(n+1) = N_3(n)+1 | \hat {\bf N}(n)\big)
				\geq \frac{1-\Delta_n}{1+\Delta_n}\cdot p_3(\hat{\bf N}(n))
				\geq (1-\varepsilon) \cdot p_3(\hat{\bf N}(n)).
			\end{equation}
			By Lemmas~\ref{lemma:inequalities_for_the_nests} and~\ref{lemma:minw1}, almost surely,
			\[\liminf_{n\to\infty} \hat N_1(n)\geq \alpha
			\quad\text{ and }\quad
			\limsup_{n\to\infty} (\hat N_1(n)-\hat N_3(n))\leq \alpha.\]
			Almost surely for all $n$ large enough,
			$\hat{\bf N}(n)\in \mathcal E_\delta$, where
			\[\mathcal E_\delta = 
			\{w\in \mathcal E \colon w_1\geq \alpha-\delta
			\text{ and }w_1-w_3\leq \alpha+\delta\}.\]
			By~\eqref{eq:truc3}, almost surely for all $n\geq n_\delta$,
			\[\mathbb P\big(N_3(n+1) = N_3(n)+1 | \hat {\bf N}(n)\big)
			\geq h(\hat N_3(n)),\]
			where $h(w_3) := (1-\varepsilon) \inf\{p_3(w)\colon w = (w_1, w_2, w_3)\in\mathcal E_\delta\}$.
			By~\eqref{eqn:defP3} and~\eqref{eqn:defp}, we have, for all $w\in\mathcal E$ (see~\eqref{eq:def_E}),
			\[p_3(w) 
			= w_3\cdot \frac{\alpha \ell_1 (1-w_1)+(1-\alpha)\ell_2w_1}{\ell_3 w_1(1-w_1)+\ell_2w_1w_3+\ell_1w_3(1-w_1)}.\]
			Now note that, for all $w\in \mathcal E_\delta$, $\alpha-\delta \leq w_1\leq \alpha+\delta+w_3$, and thus
			\[p_3(w) 
			\geq w_3\cdot \frac{\alpha \ell_1 (1-\alpha-\delta-w_3)+(\alpha-\delta) (1-\alpha)\ell_2}{\ell_3 (\alpha+\delta+w_3)(1-\alpha+\delta)+(\ell_1+\ell_2)w_3}.
			\]
			There exists $\kappa>0$ small enough so that, for all  $w\in\mathcal E_\delta$ satisfying $w_3\leq \kappa$, 
			\[p_3(w) 
			\geq (1-\varepsilon)w_3\cdot \frac{\alpha \ell_1 (1-\alpha-\delta)+(\alpha-\delta) (1-\alpha)\ell_2}{\ell_3 (\alpha+\delta)(1-\alpha+\delta)}.\]
			By~\eqref{eq:choice_eps}, we get that, for all  $w\in\mathcal E_\delta$ satisfying $w_3\leq \kappa$,
			$p_3(w)\geq (1+\varepsilon)w_3/(1-\varepsilon)$.
			This implies that, for all $w_3\leq \kappa$, 
			$h(w_3)\geq (1+\varepsilon) w_3$.
			Thus, by Corollary~\ref{cor:urnstochdomin},
			\[\liminf_{n\uparrow\infty} \hat N_3(n)\geq \kappa>0,\]
			which concludes the proof.
		\end{proof}

		\begin{lemma}\label{lemma:majw1}
			If $\ell_2 < \ell_1 + \ell_3$, then $\limsup_{n\to\infty} \hat{N}_1(n)<1$ almost surely.
		\end{lemma}
		
		\begin{proof}
			
			First note that, by definition of the model, almost surely for all $n\geq 0$, $\hat N_1(n) = 1-\hat N_2(n)$. It is thus enough to prove that $\liminf_{n\to\infty} \hat{N}_2(n)>0$ almost surely.
			By Lemma~\ref{cor:bound_P_p}, almost surely for all $n$ large enough,
			\[\frac{1-\Delta_n}{1+\Delta_n}\cdot p_2(\hat{\bf N}(n)) 
			\leq \mathbb P\big(N_2(n+1) = N_2(n)+1 | \hat {\bf N}(n)\big)
			\leq \frac{1+\Delta_n}{1-\Delta_n}\cdot p_2(\hat{\bf N}(n)),\]
			where $\Delta_n =  \sup_i \{ \ell_i K_i n^{(\alpha_i-1)\eps_i}\}$, $K_1, K_2, K_3$ are almost surely finite random variables, $\eps_1, \eps_2, \eps_3$ are three positive constants, and $\alpha_i = \ell_i/(\ell_i+1)$ for all $1\leq i\leq 3$.
			For all $\varepsilon>0$, 
			almost surely for all $n$ large enough,
			\begin{equation}\label{eq:truc_again2}
				\mathbb P\big(N_2(n+1) = N_2(n)+1 | \hat {\bf N}(n)\big)
				\geq (1-\varepsilon) p_2(\hat{\bf N}(n)).
			\end{equation}
			We choose $\delta>0$ and $\eps>0$ such that $\delta < 1-\alpha-\delta$, $\delta < 1-\beta_1-\delta$, and
			\begin{equation}
				\frac{(1-\alpha)(1-\delta) \ell_3  + (1-\alpha-\delta) \ell_1}{\ell_3\delta+ \ell_1 \delta(1-\alpha+\delta)+\ell_2(1-\alpha+\delta)}
				> \frac{1+\varepsilon}{1-\varepsilon}. \label{eqn:hyp_delta}
			\end{equation}
			This is possible since 
			\[\frac{(1-\alpha)(1-\delta) \ell_3  + (1-\alpha-\delta) \ell_1}{\ell_3\delta+ \ell_1 \delta(1-\alpha+\delta)+\ell_2(1-\alpha+\delta)}
			\to \frac{\ell_3+\ell_1}{\ell_2}>1,\]
			as $\delta\downarrow0$, 
			and because $\alpha<1$ and $\beta_1<1$ (see Remark~\ref{rem:beta1andbeta3}).
			Now, by Corollary \ref{cor:cv} and Lemma \ref{lemma:minw1}, almost surely,
			\[\lim_{n \to \infty}{\hat{\bf N}}(n) \in \{ (1,1-\alpha),(\alpha, 0), (\beta_1,\beta_3)\}.\]
			Thus, almost surely for all $n$ large enough, $\hat{\bf N}(n)\in \mathcal E_\delta$, where $\mathcal E_\delta$ is the set of all points at distance at most $\delta$ of $\{ (1,1-\alpha),(\alpha, 0), (\beta_1,\beta_3)\}$ (for the $L^\infty$-norm). 
			Thus, by~\eqref{eq:truc_again2}, almost surely for all $n$ large enough,
			\begin{equation}\label{eq:truc_again3}
				\mathbb P\big(N_2(n+1) = N_2(n)+1 | \hat {\bf N}(n)\big)
				\geq g(\hat N_2(n)),
			\end{equation}
			where 
			$g(w_2) = (1-\varepsilon)
			\inf_{w_3\colon (1-w_2, w_3)\in\mathcal E_\delta} p_2(1-w_2, w_3)$
			if $w_2$ is at distance at most $\delta$ of $\{0, 1-\alpha, 1-\beta_1\}$, and $g(w_2) = 0$ otherwise.
			Now, assume that $w_2<\delta$ and $(1-w_2, w_3)\in\mathcal E_\delta$.
			Because $\delta<1-\alpha-\delta$ and $\delta<1-\beta_1-\delta$, we get that
			$(1-w_2, w_3)$ is at distance larger than $\delta$ of both $(\alpha, 0)$ and $(\beta_1,\beta_3)$.
			Thus, it must be at distance at most $\delta$ of $(1, 1-\alpha)$, which implies
			$w_3\in(1-\alpha-\delta, 1-\alpha+\delta)$. 
			Recall that 
			\[p_2(1-w_2,w_3) = \frac{(1-\alpha) \ell_3 (1-w_2) + \ell_1 w_3 }{\ell_3w_2(1-w_2)+ \ell_1 w_2w_3+\ell_2(1-w_2)w_3}\cdot w_2.\]
			Thus, because $w_2<\delta$ and $1-\alpha-\delta < w_3 < 1-\alpha+\delta$, we get
			\[p_2(1-w_2,w_3)
			> \frac{(1-\alpha)(1-\delta) \ell_3  + (1-\alpha-\delta) \ell_1}{\ell_3\delta+ \ell_1 \delta(1-\alpha+\delta)+\ell_2(1-\alpha+\delta)}\cdot w_2
			> \frac{1+\varepsilon}{1-\varepsilon}\cdot w_2,\]
			by~\eqref{eqn:hyp_delta}.
			Therefore, for all $w_2<\delta$, $g(w_2) > (1+\varepsilon)w_2$.
			Thus, by~\eqref{eq:truc_again3} and Corollary \ref{cor:urnstochdomin}, $\liminf_{n\to\infty} \hat{N}_2(n)\geq\delta>0$, concludes the proof.
		\end{proof}

		\subsection{Conclusion}
		
		We can now conclude the proof of Theorem \ref{thm:trianglesSP}.
		
		\begin{proof}[Proof of Theorem \ref{thm:trianglesSP}] 
			We treat the three cases of Theorem \ref{thm:trianglesSP} separately:
			
			$\bullet$ We first assume that $\ell_3 \geq \ell_1+\ell_2$, which implies $\ell_2\leq \ell_3-\ell_1<\ell_1+\ell_3$ (because $\ell_1\geq 1$, by assumption).
			In this case, $(\beta_1,\beta_3) \notin [0,1]^2$ (except if $\ell_3 = \ell_1+\ell_2$, but then $(\beta_1,\beta_3) = (\alpha,0)$).
			Lemma~\ref{lemma:majw1} applies in this case and gives that, almost surely, $\hat{\bf N}(n)$ does not converge to $(1,1-\alpha)$. Also, Lemma~\ref{lemma:minw1} implies that, almost surely, $\hat{\bf N}(n)$ does not converge to $(0,\alpha)$.
			Thus, by Corollary \ref{cor:cv}, we get that, almost surely as $n$ tends to infinity,
			\[ {\hat{\bf N}}(n) \to (\alpha,0),\]
			as desired.
			
			$\bullet$ We now assume that $\ell_2 \geq \ell_1+\ell_3$, which implies $\ell_3\leq \ell_2-\ell_1<\ell_2+\ell_1$.
			In this case again, $(\beta_1,\beta_3) \notin [0,1]^2$ (except if $\ell_2 = \ell_1+\ell_3$, but then $(\beta_1,\beta_3) = (1,1-\alpha)$).
			By Lemma \ref{lemma:minw3}, we get that, almost surely, $\hat{\bf N}(n)$ does not converge to $(\alpha, 0)$. Also, by Lemma~\ref{lemma:minw1}, almost surely, $\hat{\bf N}(n)$ does not converge to $(0, \alpha)$.
			Thus, by Corollary \ref{cor:cv}, we get that, almost surely as $n$ tends to infinity,
			\[{\hat{\bf N}}(n)\to (1,1-\alpha),\]
			as claimed.
			
			$\bullet$ Finally, if $\ell_2 < \ell_1+\ell_3$ and $\ell_3 < \ell_1+\ell_2$, the same reasoning combining Lemmas \ref{lemma:minw1}, \ref{lemma:minw3}, \ref{lemma:majw1} and Corollary \ref{cor:cv} implies that, almost surely as $n$ tends to infinity,
			\[{\hat{\bf N}}(n) \underset{n\to\infty}{\longrightarrow} (\beta_1,\beta_3).\]
			
			Finally, to conclude about the edge-weights convergence, recall that, by Lemma~\ref{lemma:single_nest_in_multi_nest}, the processes induced by the two-nest process on each of $G_1$, $G_2$ and $G_3$ are three single-nest ants processes. 
			Thus, convergence of the edge-weights follows from~\cite[Theorem~1.3]{kious2020finding}, which concludes the proof.
		\end{proof}
		
		\bibliographystyle{alpha}
		\bibliography{biblio.bib}

	\end{document}